\documentclass[a4paper,11pt]{article}
\usepackage[utf8]{inputenc}
\usepackage{amsfonts}
\usepackage{amsmath}
\usepackage{amssymb}
\usepackage{amsthm}

\usepackage{authblk}

\usepackage{xcolor,a4wide}
\usepackage{algorithm,algpseudocode, algorithmicx}
\usepackage{graphicx}
\usepackage[caption=false]{subfig}
\usepackage[shortlabels]{enumitem}
\usepackage{hyperref}
\hypersetup{
    colorlinks=true,
    linkcolor=blue,
    filecolor=magenta,      
    urlcolor=cyan,
}

\usepackage{array,multirow}
\newcolumntype{B}[1]{>{\centering\arraybackslash}m{#1}}
\usepackage{colortbl}
\usepackage[normalem]{ulem} 
\usepackage[pagewise]{lineno} 

\usepackage[nocompress]{cite}

\theoremstyle{plain}

\theoremstyle{plain}

\theoremstyle{plain}
\newtheorem{assumption}{Assumption}[section]
\theoremstyle{definition}

\theoremstyle{plain}

\theoremstyle{remark}
\newtheorem{remark}{Remark}[section]
\theoremstyle{remark}

\theoremstyle{plain}
\newtheorem{proposition}{Proposition}[section]

\DeclareMathOperator*{\argmin}{\arg \min}

\newcommand{\R}{\mathbb{R}} 
 
\DeclareMathOperator{\prox}{prox}
\DeclareMathOperator{\dom}{dom}

\title{Trust your source: quantifying source condition elements for variational regularisation methods}

\date{}

\author[1]{Martin Benning}
\author[2]{Tatiana A.~Bubba}
\author[3]{Luca Ratti}
\author[4]{Danilo Riccio}
\affil[1]{Department of Computer Science, University College London, UK}
\affil[2]{Department of Mathematical Sciences, University of Bath, Bath, UK}
\affil[3]{Department of Mathematics, University of Bologna, Bologna, Italy}
\affil[4]{School of Mathematical Sciences, Queen Mary University of London, London, UK}

\begin{document}

\maketitle

\abstract{Source conditions are a key tool in  regularisation theory that are needed to derive error estimates and convergence rates for ill-posed inverse problems. In this paper, we provide a recipe to practically compute source condition elements as the solution of convex minimisation problems that can be solved with first-order algorithms. We demonstrate the validity of our approach by testing it  on two inverse problem case studies in machine learning and image processing: sparse coefficient estimation of a polynomial via LASSO regression and recovering an image from a subset of the coefficients of its discrete Fourier transform. We further demonstrate that the proposed approach can easily be modified to solve the machine learning task of identifying the optimal sampling pattern in the Fourier domain for a given image and variational regularisation method, which has applications in the context of sparsity promoting reconstruction from magnetic resonance imaging data. 
}

\section{Introduction}
The canonical example of an  inverse problem is the resolution of the linear equation
\begin{equation}\label{eq:InvProbl}
  K u^{\dag} = f  
\end{equation}
where $u^{\dag} \in X$ is the quantity we wish to recover from the data $f \in Y$, and $K: X \rightarrow Y$ is a (bounded, linear) operator. 
In order for \eqref{eq:InvProbl} to be ill-posed, the operator $K$ is usually assumed to be compact, such that the Moore-Penrose inverse of $K$ is unbounded. In most practical applications, the exact data $f$ is not accessible; instead, we only have access to a perturbed version $f^{\delta} \in Y$ and we are required to combat the ill-posedness with regularisation methods. Over the past decades, variational regularisation methods have been extensively used to approximate the solution of ill-posed inverse problems~\cite{scherzer2009variational,benning2018modern}. The key idea of variational regularisation methods consists of approximating $u^{\dag}$ in~\eqref{eq:InvProbl} as the solution of the minimisation problem
\begin{equation}\label{eq:VarRegu}
    \argmin_{u \in X} \left\{\mathcal{D}(Ku, f^{\delta}) + \alpha J(u) \right\},
\end{equation}
where $\mathcal{D}(Ku, f^{\delta})$ is a data fidelity term expressing the gap between the computed and measured data, $J(u)$ is a regularisation functional favouring a-priori information  on the solution or, equivalently, penalising solutions with undesired structures, and $\alpha >0$ is the regularisation parameter controlling the influence of the two terms on the minimiser. In this work, the data fidelity term is simply the least-squares term $\mathcal{D}(Ku, f^{\delta}) = \frac{1}{2} \|Ku-f^{\delta}\|_Y^2$. The choice of the regularisation term is less straightforward and it has been an active field of research: from the pioneering work of Tikhonov~\cite{tikhonov1943stability, tihonov1963solution}, to more recent papers on sparsity promoting regularisation~\cite{donoho1992superresolution,candes2006robust,daubechies2004iterative}, passing through the seminal work of Rudin et al.~\cite{rudin1992nonlinear} and Mumford and Shah~\cite{mumford1989optimal} where nonlinear variational models were first introduced. We want to emphasise that the concepts of inverse problems and variational regularisation also play a key role in machine learning tasks such as regression and logistic regression, with the LASSO \cite{tibshirani1996regression} being one of the most prominent examples of a sparsity-promoting regression approach of the form of \eqref{eq:VarRegu}.

One reason for the popularity of variational regularisation methods is that, along with an intuitive approach to modelling, these provide a framework for their basic analysis, particularly in the case of convex regularisation functionals. In this context, once the existence of a minimiser for~\eqref{eq:VarRegu} is guaranteed by suitable assumptions, a compelling matter is establishing error estimates and quantify rates of convergence. These are generally smoothness conditions on the solution that are required, which go under the name of \textit{source conditions}.

There exist several different types of source conditions and a detailed survey of their role  to derive convergence rates can be found in~\cite[Chapter 3]{schuster2012regularization}. Classical source conditions for nonlinear variational methods trace back to \cite{chavent1997regularization} and the seminal work of Burger and Osher~\cite{burger2004convergence}, where the authors derive convergence rates in the Bregman distance~\cite{bregman1967relaxation} setting for general convex functionals in Banach spaces, extending the results already known in the quadratic case~\cite{engl1989convergence}. 
By that time, convergence rates under source conditions were established in the Hilbert setting~\cite{tautenhahn1998optimality}, also for the nonlinear case~\cite{engl1996regularization}.

An alternative -- but often equivalent (see \cite[Theorem 5.12]{benning2018modern}) -- formulation to classical source conditions is provided by \textit{range conditions} (see~\cite[Proposition 1]{burger2004convergence} and \cite[Definition 5.8]{benning2018modern}), which ensure that the unknown solution $u^{\dag}$ of the inverse problem \eqref{eq:InvProbl} is in the range of the variational regularisation operator. For example, range conditions (together with a restricted injectivity assumption) allow to establish linear convergence rates for $\ell^1$ regularisation, avoiding stronger assumptions like the restricted isometry property~\cite{grasmair2011necessary}. More recently, range conditions have also been deployed in the context of data-driven regularisation~\cite{mukherjee2021learning}.

While we will focus on source and range conditions in this work, we want to emphasise that stronger source conditions exist (like those in~\cite{resmerita2005regularization}) that yield better convergences 
rates compared to those derived in~\cite{burger2004convergence}. 

Other noticeable source conditions formulated as a variational inequality (instead of an equation or inclusion) are the so-called variational source conditions first introduced in~\cite{hofmann2007convergence} in the context of nonlinear ill-posed inverse problems. These have been widely used to derive convergence rates~\cite{flemming2012generalized}: in particular, in the context of sparsity promoting regularisers~\cite{grasmair2008sparse,hohage2017characterizations,hohage2019optimal}. Last but not least, approximate source conditions have been used to quantify errors of regularisation methods and were originally introduced in the Hilbert setting~\cite{hofmann2006approximate}, but quickly adapted to the Banach space setting~\cite{hein2009approximate}. They are generally weaker compared to classical source conditions, but allow to measure effectively how well the range condition can be approximated~\cite{burger2018large}.

Over the years, source conditions have been regarded as a purely theoretical tool to quantify convergence rates. In the papers cited above, there is little or no discussion on practical problems or applications where source conditions are proven to be met. Even for simple (non-trivial) inverse problems, systematic numerical studies of convergence rates for which the solutions provably verify a source condition are rare. Such examples include~\cite{ramlau2010convergence}, where the authors consider a convolution operator and show that the convergence rates derived in their work are only verified in the case of a solution fulfilling the source condition. In~\cite[Example 6.1 \& Section 6.2.1]{benning2011error}, the authors compare their theoretical error estimates with the computational errors. In order to do so, they propose two inverse problem solutions of which one can be shown to analytically satisfy the source condition while the other violates it (Example 6.1), and numerically estimate a source condition element to verify error estimates (Section 6.2.1.). More recently, in the context of statistical inverse  problems, in~\cite{bubba2021convex,bubba2022shearlet} the authors derive convergence rates in expected Bregman distance, using both classical and approximate source conditions. These are numerically verified for a tomographic problem using a solution that fulfils the source condition. Another exception is the special case of generalised eigenfunctions and singular vectors, whose theoretical and numerical computation has attracted significant interest over previous years (cf. \cite{benning2013ground,gilboa2014nonlinear,gilboa2014total,burger2015spectral,burger2016spectral,gilboa2016nonlinear,schmidt2018inverse,benning2018modern,gilboa2018nonlinear,nossek2018flows,bozorgnia2020infinity,bungert2021nonlinear} and references therein).

Motivated by the lack of a generic, practical strategy to reliably estimate source and range condition elements, 
in this paper \textbf{we propose a novel approach to compute source condition elements as the solution of a convex and differentiable minimisation
problem}. We prove that the solution is equivalent to fulfilling a classic source (or range) condition when $K$ in~\eqref{eq:InvProbl} is injective. We demonstrate the validity of our strategy by testing it numerically for two case studies: the machine learning problem of sparse polynomial regression and the image processing problem of recovering an image from a subset of its Fourier samples. Having at disposal explicit knowledge of the source condition elements has indubitable significance from a theoretical perspective: 1) the source condition element provides exact quantitative error estimates (i.e., the constants appearing in the convergence rates bounds can be computed explicitly), and 2) computing a range condition element allows us to determine the data that is required for a variational regularisation approach to  approximate the solution of an inverse problem. However, it also opens up interesting lines of research in the context of learning for inverse problems. In this work, we show how estimating a sparse source condition element can be used to devise a strategy for optimally sampling in the Fourier domain in the context of variational regularisation. Applications for such a problem can for instance be found in Magnetic Resonance Imaging (MRI) (cf. \cite{sherry2020learning}).

\textit{Our contributions}. 
The main contribution of this paper is a novel, practical approach to compute source conditions elements, with the aim of depicting the source conditions as quantitative tools in variational regularisation, rather than merely theoretical requirements.
To do so, we consider a rather general Banach space setup, and we employ tools from convex analysis to provide an alternative formulation for the source condition, together with an iterative procedure for the numerical approximation of the source condition element. 
To demonstrate the potential of this strategy in  applications, we provide a set of numerical experiments where we focus on two ill-posed inverse problems: 1D interpolation and 2D imaging. Finally, we show how a slight modification of the proposed approach forms a supervised optimal design technique by suitably identifying samples in the Fourier domain to promote desired properties of the associated source condition elements.

\textit{Structure of the paper}. The remainder of the paper is organised as follows. Section~\ref{sec:MathsPrelim} is devoted to briefly reviewing relevant theoretical concepts from convex analysis. In Section~\ref{sec:part-a}, we describe what source conditions are, why they are necessary and detail the key idea of our approach, namely, how to compute source condition elements as the solution of a convex minimisation problem. We demonstrate the performance of our methodology by a series of numerical experiments in Section~\ref{sec:numerics} with focus on machine learning. In particular, in Section~\ref{sec:optimal-sampling} we show how our strategy can be used to  obtain the optimal sampling pattern in the Fourier domain for a fixed image and given variational regularisation method. Concluding remarks and future prospects are summarised in Section~\ref{sec:conclusions}.

\section{Mathematical preliminaries}\label{sec:MathsPrelim}
In this section, we set the notation and give a brief overview of some key results from convex analysis which are useful for the rest of the paper (see \cite{ekeland1999convex}). 

For a functional $F \colon X \rightarrow \R$, the subdifferential evaluated at a point $\hat{u} \in X$ is defined as
\[
\partial F(\hat{u}) = \left\{x^* \in X^*: F(u)-F(\hat{u}) \geq \langle x^*, u-\hat{u} \rangle_* \, ,  \: \forall u\in X \right\} ,
\]
where $\langle \cdot, \cdot \rangle_*$ denotes the canonical pairing between the dual space $X^*$ and $X$.
If $F$ is also differentiable, the only element of the subdifferential $\partial F(\hat{u})$ is the Fr\'{e}chet-derivative $\nabla F(\hat{u})$. \\
The convex conjugate of a functional $F \colon X \rightarrow \R$ is defined as $F^\star \colon X^* \rightarrow \R$ such that
\[
F^\star \colon p \in X^* \mapsto F^\star(p) = \sup_{ u \in X}\left\{ \langle p,u \rangle_* - F(u) \right\}.  
\]

If $X$ is a Hilbert space, we can interpret $F^\star$ as a functional from $X$ to $\R$, by identifying $X$ with $X^*$ via the Riesz's representation theorem. \\
For a continuous convex functional $F$, the Bregman distance between $u$ and $w$, associated with $p \in \partial F(w)$ is defined as
\[
D_F^p(u,w) = F(u) - F(w) - \langle p, u-w \rangle_*.
\]
Such a mapping is actually not a distance: it is non-negative but, in general, it is not symmetric and it does not satisfy the  triangle inequality. To recover the symmetry, we usually rely on the symmetric Bregman distance, also known as the Jeffrey's distance, namely:
\[ 
D_F^{\text{symm},q,p}(u, w) = D_F^{q}(w,u) + D_F^{p}(u, w), \qquad q \in \partial F(u), \ p \in \partial F(w).
\] 
Notice that if $F$ is differentiable the subdifferentials are single-valued, hence we can drop the dependence  of $p,q$ and simply denote by $D_F(u,w)$ the symmetric Bregman distance between $u$ and $w$.

\section{Computing source condition elements}\label{sec:part-a}

We consider the inverse problem \eqref{eq:InvProbl} of recovering $u^\dag \in X$ from $f^\delta \in Y$, a perturbation of the noiseless measurements $f = K u^\dag$ satisfying $\| f^\delta - f \|_{Y} \leq \delta$ for some given noise level $\delta \geq 0$.
The forward map $K \colon X \rightarrow Y$ is a bounded linear operator, but we allow it to be compact and hence,  do not assume that it admits a bounded inverse in general. As a consequence, retrieving $u^\dag$ from $f^\delta$ is an ill-posed problem, and solving \eqref{eq:InvProbl} can suffer from instability in particular. The goal of regularisation theory for inverse problems is to construct a (family of) continuous operators that provide a satisfactory approximation of the potentially discontinuous inverse map. One of the most prominent paradigms is represented by variational regularisation, where a family of (potentially set-valued) operators $R_\alpha$, parameterised by $\alpha \in (0, \infty)$, is defined as
\begin{align}
R_\alpha\colon Y \rightrightarrows X, \quad R_\alpha \colon f \mapsto u_\alpha \in \argmin_{u \in X} \left\{ \frac{1}{2}\|Ku-f\|_Y^2 + \alpha J(u) \right\},\label{eq:variational-regularisation}
\end{align}
with $J\colon X \rightarrow \R \cup \{ \infty \}$ being a so-called regularisation functional. For the remainder of this work, we make the following set of assumptions.
\begin{assumption} 
In line with \cite{benning2018modern}, we assume that the following conditions are satisfied:
\begin{itemize}
    \item $Y$ is a Hilbert space;
    \item $X$ is the dual of a normed space $\mathcal{X}$ such that its weak-star topology on $X$ is metrisable on bounded sets;
    \item $K$ is the adjoint of a bounded linear operator from $Y$ to $\mathcal{X}$;
    \item $J$ is the conjugate of a proper functional from $\mathcal{X}$ to $\R \cup \{\infty\}$ and is non-negative;
    \item for every $f \in Y$ and $\alpha > 0$ there exists a constant $c = c(a, b, \| f \|)$ that depends monotonically non-decreasing on all arguments such that $\|u\|_X \leq c$ if $\|Ku - f \| \leq a$ and $J(u) \leq b$.
\end{itemize} \label{ass:1}
\end{assumption}
Notice in particular that the assumption on $J$ implies that it is a convex functional. Under those assumptions, it is possible to prove that \eqref{eq:variational-regularisation} is well defined: namely, that for every $f \in Y$ and $\alpha >0$ the set $R_\alpha(f)$ is non-empty (see \cite[Theorem 5.6]{benning2018modern}, which also shows that $R_\alpha(f)$ is convex). Moreover, for $\alpha > 0$, each operator $R_\alpha$ is stable in the sense of the Kuratowski limit superior: namely, for every sequence $f_n \rightarrow f$ in $Y$ there exists a subsequence
$u_{n_k} \in R_\alpha(f_{n_k})$ converging to an element $u^* \in R_\alpha(f)$ in the weak-star topology on $X$ (see \cite[Theorem 5.7]{benning2018modern}).
To conclude that $R_\alpha$ is a family of regularisers, one needs to ensure that, if the noise level $\delta$ converges to zero, i.e., $\delta  \searrow 0$, there exists a choice $\alpha(\delta)$ such that $R_{\alpha(\delta)}(f^\delta)$ converges to $u^\dag$. More advanced results in this direction also provide convergence rates, usually at the price of an additional assumption involving the unknown solution $u^\dag$. In particular, an important  tool to obtain quantitative error estimates between the solution $u^\dagger$ and the regularized solution $u_\alpha$ is represented by  \textit{source conditions}. In their most classical formulation, they require to assume that there exists $v \in Y$ such that
\begin{align}
    K^\ast v \in \partial J(u^\dagger), \label{eq:sc}\tag{SC}
\end{align}
where $K^\ast\colon Y \rightarrow X^*$ is the adjoint of $K$ (thanks to Riesz representation lemma, we identify $Y$ with its dual), and $\partial J \subset X^*$ is the subdifferential of $J$ (see Section \ref{sec:MathsPrelim}).
The object $v \in Y$ is referred to as the source condition element: in the following, we want to recall why elements satisfying the source condition  allow us to derive error estimates with convergence rates  (Section \ref{sec:error-estimates}) and then discuss how to compute $v$ for given $u^\dagger$, if such $v$ exists, in Section \ref{sec:sc-minimisation}.

\subsection{Error estimates for variational regularisation methods}\label{sec:error-estimates}

In \cite{benning2018modern,burger2004convergence}, error estimates measured in Bregman distances have been established for convex but non-smooth variational regularisation methods with the help of source conditions. For reasons of self-containment, we recall one of the key results. 

We start by an equivalent formulation of \eqref{eq:variational-regularisation} by means of first-order optimality conditions, exploiting the convexity of $J$: namely, $u_\alpha \in R_\alpha(f^\delta)$ if and only if
\[
    \exists \, p_\alpha \in \partial J(u_\alpha): \quad  K^\ast( K u_\alpha - f^\delta ) + \alpha p_\alpha = 0,
\]
where the previous equality holds in $X^*$. Assuming that \eqref{eq:sc} is satisfied, we can subtract $\alpha K^\ast v$ on both sides of the equation to obtain
\[
    K^\ast( K u_\alpha - f^\delta ) + \alpha \left(  p_\alpha - K^\ast v \right) = - \alpha K^\ast v .
\]    
Taking a dual product of both sides of this equation with $u_\alpha - u^\dagger$ then yields
\[
    \langle Ku_\alpha - f^\delta, Ku_\alpha - f \rangle_Y + \alpha \langle u_\alpha - u^\dagger, p_\alpha - K^\ast v \rangle_* = - \alpha \langle K^\ast v, u_\alpha - u^\dagger \rangle_* \,.
\]
Please note that $\langle\cdot,\cdot \rangle_Y$ indicates the inner product in $Y$, whereas $\langle\cdot,\cdot \rangle_*$ denotes the pairing between a functional in $X^*$ and an element of $X$.
We can interpret the second term on the left-hand side by means of the symmetric Bregman distance. Since $p_\alpha \in \partial J(u_\alpha)$ and $K^\ast v \in \partial J(u^\dagger)$, the quantity $\langle u_\alpha - u^\dagger, p_\alpha - K^\ast v \rangle_*$ equals $D_J^{\text{symm},p_\alpha,K^*v}(u_\alpha,u^\dag)$ (in the following we only refer to this term as $D_J^{\text{symm}}$ for ease of notation). In addition, we can reformulate $\langle Ku_\alpha - f^\delta, Ku_\alpha - f \rangle_Y$  as \[
\langle Ku_\alpha - f + f - f^\delta, Ku_\alpha - f \rangle_Y = \frac12 \| Ku_\alpha - f \|_Y^2 + \frac12 \| Ku_\alpha - f^\delta \|_Y^2 - \frac12 \| f - f^\delta \|_Y^2,
\]
to obtain
\[
    \frac12 \| Ku_\alpha - f \|_Y^2 + \frac12 \| Ku_\alpha - f^\delta \|_Y^2 + \alpha D_J^{\text{symm}}(u_\alpha, u^\dagger) = \frac12 \| f - f^\delta \|_Y^2 - \alpha \langle v, Ku_\alpha - f \rangle_Y ,
\]
where we have made use of \eqref{eq:InvProbl}, which assumes the identity $Ku^\dagger = f$. Using the identity 
\[
\langle \alpha  v, f - Ku_\alpha \rangle_Y = \frac{\alpha^2}{2} \| v \|_Y^2 + \frac12 \| Ku_\alpha - f \|_Y^2 - \frac12 \| \alpha v - f + Ku_\alpha \|_Y^2
\]
and subsequently eliminating $\frac12 \| Ku_\alpha - f \|_Y^2$ on both sides leads to the equation 
\[
    \frac12 \| Ku_\alpha - f + \alpha v \|_Y^2 + \frac12 \| Ku_\alpha - f^\delta \|_Y^2 + \alpha D_J^{\text{symm}}(u_\alpha, u^\dagger) = \frac12 \| f - f^\delta \|_Y^2 + \frac{\alpha^2}{2} \| v \|_Y^2 .
\]
Dividing by $\alpha > 0$ and using the estimate $\| f - f^\delta \|_Y \leq \delta$ then yields the well-known error estimate (cf. \cite{burger2007error})
\[
    \frac{1}{2\alpha} \| Ku_\alpha - f + \alpha v \|_Y^2 + \frac{1}{2\alpha} \| Ku_\alpha - f^\delta \|_Y^2 + D_J^{\text{symm}}(u_\alpha, u^\dagger) \leq \frac{\alpha}{2} \| v \|_Y^2 + \frac{\delta^2}{2 \alpha} .
\]
With the a-priori choice $\alpha(\delta) = \delta / \| v \|_Y$ we minimise the right-hand-side of this inequality and obtain
\begin{align}
\begin{split}
    0 &\leq \frac{\| v \|_Y}{2\delta} \left\| Ku_{\alpha(\delta)} - f + \delta \frac{v}{\| v \|_Y} \right\|_Y^2 + \frac{\| v \|_Y}{2\delta} \| Ku_{\alpha(\delta)} - f^\delta \|_Y^2 + D_J^{\text{symm}}(u_{\alpha(\delta)}, u^\dagger) \\
    &\leq \| v \|_Y \delta .
    \end{split}\label{eq:error-estimate-linear}
\end{align}
Hence, the symmetric Bregman distance between $u_{\alpha(\delta)}$ and $u^\dagger$ is bounded by the worst-case data error bound $\delta$ multiplied by the norm of the source condition element $v$. To quantify a-priori error estimates such as  \eqref{eq:error-estimate-linear} it is therefore vital to quantify the norm of the source condition element.

\subsection{Casting the computation of source condition elements as convex minimisation problems}\label{sec:sc-minimisation}
Having established the need for source condition elements and estimates for their norm in order to quantify a-priori error estimates of regularised inverse problems solutions, we now discuss how we can formulate the computation of $v$ verifying \eqref{eq:sc} as a variational minimisation problem for a suitable convex and differentiable functional.

We first assume, in analogy to \cite{burger2007error}, that the space $X$ is continuously embedded into a Hilbert space $Z$. By means of the Hahn-Banach theorem, $K$ can be extended to a bounded linear operator from $Z$ to $Y$ (which, with an abuse of notation, we still denote as $K$), whereas we interpret the functional $J$ as acting on the whole space $Z$ by extending it to $\infty$ outside of $X$. By this modification, we can equivalently solve the minimisation problem \eqref{eq:variational-regularisation} in $X$ or in $Z$.

\begin{remark}
To rigorously verify the last statement, assume that $\tilde{K}\colon Z \rightarrow Y$ and $\tilde{J} \colon Z \rightarrow \R$ are the proposed extensions of $K$ and $J$. Then, for any $f \in Y$, the problem
\[
\tilde{u}_\alpha \in \tilde{R}_\alpha(u) = \argmin_{u \in Z} \left\{\frac{1}{2}\|\tilde{K}u - f \|_Y^2 + \alpha \tilde{J}(u)\right\}
\]
is such that $\tilde{R}_\alpha(f) = R_\alpha(f)$. Moreover, the optimality conditions are equivalent, since it is possible to show that, for $\hat{u} \in X$, the set $\partial \left\{\frac{1}{2}\|\tilde{K}\cdot - f \|_Y^2 + \alpha \tilde{J}\right\} (\hat{u}) \subset Z$ can be identified (via bounded extension) with the set $\partial \left\{\frac{1}{2}\|K\cdot - f \|_Y^2 + \alpha J\right\} (\hat{u}) \subset X^*$.
\end{remark}

Motivated by \cite{wang2022lifted}, we provide a characterisation of the source condition element involving the use of proximity operators. For a proper, convex, lower semi-continuous functional $F$ defined on the Hilbert space $Z$, the proximity operator (or proximal map) is defined as
\begin{equation}
    \prox_F\colon Z \rightarrow Z, \qquad \prox_F(z) := \argmin_{u \in Z} \left\{ \frac12 \| u - z \|_Z^2 + F(u) \right\}\,.
    \label{eq:proximal}
\end{equation}
The Hilbert space structure is, in principle, not needed for the definition of the proximal operator: for example, if $X$ is assumed to be a Banach space, its norm can be used inside \eqref{eq:proximal}, which would not lead to an equivalent operator. We nevertheless rely on definition \eqref{eq:proximal} and on the use of the norm $\| \cdot \|_Z$. The following result provides an alternative formulation of \eqref{eq:sc}.

\begin{proposition} \label{prop:mod-sc}
The source condition \eqref{eq:sc} can be rewritten as 
\begin{equation}
    u^\dagger = \prox_J\left(u^\dagger + K^\ast v\right) ,
\label{eq:mod-sc}
\end{equation}

\end{proposition}

\begin{proof}
By Assumption \ref{ass:1}, the functional $J$ is convex, which implies that $\frac{1}{2}\| \cdot - z \|_Z^2 + J$ is strictly convex, for every $z \in Z$. Therefore, its minimiser can be determined by means of first-order optimality conditions, yielding
\begin{equation}
u = \prox_J(z) \quad \Leftrightarrow \quad z-u \in \partial J(u).
\label{eq:prox_eq}    
\end{equation}
Choosing $u = u^\dag$ and $z = u^\dag + K^* v$ (since $K$ has been extended to an operator from $Z$ to $Y$, it holds $K^*\colon Y \rightarrow Z$), we conclude \eqref{eq:mod-sc}.

\end{proof}

The advantage of the Hilbert space structure mainly resides in the possibility to use the expression \eqref{eq:prox_eq}. In Banach spaces, a similar result holds, but it requires the use of duality mappings, which would prevent the development of the following results.
From this moment on, for the ease of notation, we denote by $\|\cdot\|$ the norm on the Hilbert space $Z$, previously denoted as $\| \cdot \|_Z$, and as $\langle \cdot, \cdot \rangle$ the inner product in $Z$.

We now reformulate the source condition by means of the following functional, which, following \cite{wang2022lifted}, we refer to as the \textit{Bregman loss}:
\begin{align}
    B_J(u,p) := \left( \frac12 \| \cdot \|^2 + J\right)(u) + \left( \frac12 \| \cdot \|^2 + J\right)^\star(p) - \langle p, u \rangle  ,\label{eq:bregman-objective}
\end{align}
where $( \frac12 \| \cdot \|^2 + J )^\star$ denotes the convex conjugate of $\frac12 \| \cdot \|^2 + J$.

We provide a summary of results related to $B_J$, which provide possible interpretations as well as important properties which will be used later. A proof of the following proposition can be found in \cite[Theorem 10]{wang2022lifted}.
\begin{proposition}
The functional $B_J \colon Z \times Z \rightarrow \R$ satisfies the following properties:
\begin{enumerate}[a)]
    \item $B_J(u,p) = D_{\frac{1}{2}\| \cdot\|^2+J}^p(u, \prox_J(p))$;
    \item $B_J(u,p) = \frac{1}{2}\| u - \prox_J(p)\|^2 + D_J^{p-\prox_J(p)}(u,\prox_J(p))$;
    \item $B_J$ is bi-convex (separately in the variables $u,p$);
    \item for any fixed $u \in \dom(J)$, the function $B_J(u,\cdot)$ is continuously Fr\'{e}chet-differentiable: in particular,
    \[
    \nabla_p B_J(u,p) = \prox_J(p) - u;
    \]
    \item if moreover $\partial J(u) \neq \varnothing$, then $u = \prox_J(p)$ is a global minimiser of $B_J(u,\cdot)$.
\end{enumerate}
\label{prop:BregLoss}
\end{proposition}
\begin{remark}
Note that, if $u \in \dom(J)$ but $J(u)=\varnothing$, an element $\hat p$ such that $\nabla_p B_J\left(u,\hat p\right) = 0$ may not exist. This scenario never occurs if $u \in \operatorname{int}(\dom(J))$, and in particular when the domain of $J$ is a closed subset of $Z$. 
\end{remark}

In view of Proposition \ref{prop:BregLoss}, we focus on the problem of minimising $B_J(u,\cdot)$ for a fixed $p$. According to $a)$ and $b)$, this accounts to finding an object $\hat{p}$ such that $\prox_J(\hat{p})$ is close to $u$ and in particular, if the minimum is attained that $u = \prox_J(\hat{p})$. Moreover, thanks to the convexity and Fr\'{e}chet-differentiability of $B_J(u,\cdot)$, its minimisation of $B_J(u,\cdot)$ can be performed by means of first-order optimisation methods.

\noindent Combining ideas from Propositions \ref{prop:mod-sc} and \ref{prop:BregLoss}, we introduce the functional
\[
G_J \colon Y \rightarrow \R, \quad G_J(v) = B_J(u^\dag,u^\dag+K^*v).
\]
Since $G_J$ is a composition of an affine map with $B_J(u^\dag,\cdot)$, it inherits the convexity and differentiability, and in particular
\begin{equation}
\nabla G_J(v) = K \prox_J(u^\dagger + K^* v)- Ku^\dagger.
    \label{eq:gradient}
\end{equation}

\noindent We immediately deduce the following result, related to the source condition:
\begin{proposition} \label{prop:sc3}
If there exists $v$ which satisfies \eqref{eq:sc}, then such $v$ is a global minimiser of $G_J$. Viceversa, if $v$ is a global minimiser of $G_J$, it satisfies

\begin{equation}
 K u^\dagger = K \prox_J\left(u^\dagger + K^\ast v\right),
    \label{eq:sc_kernel}
\end{equation}
which automatically implies \eqref{eq:sc} if $K$ only has a trivial null space.
\end{proposition}

\begin{remark}
Please note that Proposition \ref{prop:sc3} does not guarantee the existence of a global minimiser of $G_J$.

\end{remark}

\begin{remark}
    We want to emphasise that if $K$ is not injective and has a non-trivial null space, an element $v$ that sets \eqref{eq:gradient} to zero can still satisfy \eqref{eq:sc} as long as $\prox(u^\dag + K^\ast v) - u^\dag$ does not lie in the null space of $K$. Furthermore, we can always verify if $v$ satisfies \eqref{eq:sc} a-posteriori.
\end{remark}

In order to check if $u^\dagger$ satisfies the source condition, and to compute the associated $v$, we are interested in solving
\begin{align*}
    \hat v = \argmin_{v} G_J(v) = \argmin_{v} B_J\left(u^\dagger, u^\dagger + K^\ast v\right) .
\end{align*}
We can solve this minimisation problem with various different first-order methods, the simplest one being gradient descent. Thanks to \eqref{eq:gradient}, the iterative step in this case reads 
\begin{align} \label{eq:graddesc}
    v^{k + 1} = v^k - \tau K \left( \text{prox}_J\left( u^\dagger + K^\ast v^k \right) - u^\dagger \right) ,
\end{align}
which is globally convergent for $\tau \leq 1/\| K\|^2$ for arbitrary initial value $v^0 \in Y$, since every proximal map is $1$-Lipschitz.
\par

The descent algorithm described in \eqref{eq:graddesc} allows to approximate a source condition element by means of simple iterations. Nevertheless, it requires the knowledge of the proximal map $\prox_J$, which has a closed-form solution for many popular choices of the functional $J$. However, many interesting functionals do not have closed-form proximal maps, such as the popular class of functionals of the form $J(u) = H(Au)$, where $H$ is a functional with closed-form proximal map but where $A$ is a linear operator. In the next section, we provide a strategy to deal with this larger family of regularisation functionals. An alternative approach, at least from a numerical perspective, is to approximate the evaluation of $\prox_J$ by means of a suitable optimisation sub-routine.

\subsection{Extension to more general functionals and range conditions}\label{sec:rc-condition}

In this section, we consider composite functionals of the form
\begin{equation}
J(u) = H(Au + b),
    \label{eq:composition}
\end{equation}
with $b\in Z$ and $A\colon Z \rightarrow Z$, and for which we assume that $\prox_H$ is known. This allows to take into account a large family of regularisation strategies, including generalised Tikhonov, sparsity promotion with respect to orthogonal bases or frames, and Total Variation, as we are going to recall in Section \ref{subsec:fourier}.
\par
In order to estimate source condition elements in this scenario, we start by relating the source condition with the \textit{range condition}. Referring to \cite[Definition 4.9]{benning2018modern}, we say that $u^\dag$ satisfies the range condition for $\alpha>0$ if there exists data $g_\alpha \in Y$ such that 
\begin{align}
    u^\dagger \in \argmin_{u \in X} \left\{ \frac{1}{2}\| Ku - g_\alpha \|^2 + \alpha J(u) \right\} ,\tag{RC1}\label{eq:range-condition}
\end{align}
which means that $u^\dag$ is in the range of the regulariser $R_\alpha$, which motivates the name of the assumption. Thanks to Assumption \ref{ass:1} and to the definition of $R_\alpha$ in \eqref{eq:variational-regularisation}, we can apply \cite[Theorem 5.12]{benning2018modern} and conclude that $u^\dagger$ satisfies the range condition \eqref{eq:range-condition} for all $\alpha > 0$ if and only if it verifies the source condition \eqref{eq:sc}. In particular, considering the equivalent optimality condition associated with \eqref{eq:range-condition}, 
\begin{align}
    \frac{1}{\alpha} K^\ast \left( g_\alpha - Ku^\dagger \right) \in \partial J(u^\dagger) , \tag{RC2}
    \label{eq:range-condition-optimality}
\end{align}
we immediately obtain the following expression connecting the source condition element $v$ with the pre-image $g_\alpha$:
\begin{equation}
g_\alpha = K u^\dag + \alpha v.
    \label{eq:sc_rc}
\end{equation}
Even though this equivalence holds true for any feasible $J$, it is particularly useful when applied to functionals of the form \eqref{eq:composition}. Indeed, in this case

the range condition \eqref{eq:range-condition-optimality} can be written as
\begin{align*}
    K^\ast(Ku^\dagger - g_\alpha) + \alpha A^\ast q^\dagger = 0 ,
\end{align*}
for $q^\dagger \in \partial H(Au^\dagger + b)$. Multiplying by $1/\alpha$ and using \eqref{eq:sc_rc}  we obtain the conditions
\begin{subequations}
\begin{align}
    K^\ast v &= A^\ast q^\dagger ,\label{eq:dual-rc1}\\
     q^\dagger &\in \partial H(A u^\dagger + b) .\label{eq:dual-rc2}
\end{align}\label{eq:dual-rc}
\end{subequations}
In analogy to Section \ref{sec:sc-minimisation}, we can rewrite \eqref{eq:dual-rc2} as $A u^\dagger + b + q^\dagger \in \partial \left( \frac12 \| \cdot \|^2 + H\right)(Au^\dagger + b)$, respectively $Au^\dag + b = \prox_H(Au^\dag + b + q^\dag)$, set up a Bregman loss functional and define the corresponding functional $G_H\colon Y \rightarrow \R$ as
\begin{align*}
    G_{H}(q) = B_{H}(Au^\dagger + b, q + Au^\dagger + b) .
\end{align*}
Then, we can formulate estimating the source condition element $v$ and the subgradient $q^\dagger$ as the minimisation of the objective functional
\begin{align}
    E_{H}(v, q^\dagger) = \frac12 \| K^\ast v - A^\ast q^\dagger \|^2 + G_{H}(q^\dagger) .\label{eq:rc-minimisation}
\end{align}
Indeed, setting the partial Fr\'{e}chet derivatives of \eqref{eq:rc-minimisation} to zero yields
\begin{align*}
    K \left( K^\ast v - A^\ast q^\dagger \right) &= 0 , \\
    A \left( A^\ast q^\dagger - K^\ast v \right) + \text{prox}_H\left(Au^\dagger + b + q^\dagger \right) - Au^\dagger - b &= 0 .
\end{align*}

The above expressions are not equivalent to the range condition \eqref{eq:range-condition}. Nevertheless, as in Proposition \ref{prop:sc3}, if we assume that $K$ has a trivial null space, then the first expression implies that \eqref{eq:dual-rc1} is satisfied, whereas the second one implies \eqref{eq:dual-rc2}.
Moreover, we can always check a-posteriori if $v$ and $q^\dagger$ satisfy \eqref{eq:dual-rc}.

In order to approximate a minimiser of $E_{H}$, we again can use first order methods. Since minimising \eqref{eq:rc-minimisation} is a minimisation problem in two variables, it seems logical to use algorithms such as explicit coordinate descent \cite{beck2013convergence,wright2015coordinate}, i.e., 
\begin{equation}
\begin{aligned}
v^{k+1} &= v^k - \tau K(K^* v^k - A^* q^k) ,\\
q^{k+1} &= q^k - \sigma\Big( A(A^* q^k - K^* v^{k+1}) + \prox_H(Au^\dag + b + q^k) - A u^\dag - b \Big),
    \label{eq:coord_desc}
\end{aligned}
\end{equation}
which guarantees the (global) convergence to a minimiser, provided that the the step sizes $\sigma,\tau$ are suitably chosen. Alternatively, one can formulate the augmented Lagrangian
\begin{align*}
    \mathcal{L}_\delta(v, q^\dagger; \mu) = \frac{\delta}{2} \| K^\ast v - A^\ast q^\dagger \|^2 + G_{H}(q^\dagger) + \langle \mu, K^\ast v - A^\ast q^\dagger \rangle 
\end{align*}
and employ algorithmic approaches such as the alternating direction method of multipliers (ADMM) \cite{lions1979splitting} and variants to find the saddle-point, or directly formulate the saddle-point problem
\begin{align*}
    \inf_{v, q^\dag} \sup_{\mu} G_{H}(q^\dag) + \langle \mu, K^\ast v - A^\ast q^\dag \rangle 
\end{align*}
and compute the saddle-point with algorithms such as the primal-dual hybrid gradient method or variants of it (see \cite{chambolle2016introduction,benning2021bregman} and references therein).

In the following section, we provide some numerical examples to discuss the effectiveness of the reconstruction of source condition elements, both via the minimisation of $G_J$ or of $E_{H}$, according to the different choices of $J$. 

\section{Numerical results}\label{sec:numerics}

In this section, we demonstrate the validity of our strategy by testing it for two case studies. For each of our experiments, we present two examples. One, where we assume that the source condition is satisfied with a corresponding source condition element with small norm. And one, where we assume that the source condition is either not satisfied or only satisfied with a corresponding source condition element with large norm.

In Subsection~\ref{subsec:1Dpoly} we consider the machine learning problem of polynomial LASSO regression, while in Subsection~\ref{subsec:fourier} we consider the problem of estimating a two-dimensional image from a subset of its Fourier samples. Finally, in Subsection~\ref{sec:optimal-sampling} we show how our strategy can be used to learn the optimal sampling pattern in the Fourier domain for a fixed image and given variational regularisation method.

\subsection{1D example: Polynomial LASSO regression}
\label{subsec:1Dpoly}

As a first example, we show how to estimate source condition elements for a polynomial regression model. We start by considering the polynomial
\begin{align}
\label{eq:results-fx-1d-polynomial-regression_deg5}
    \varphi(u) = 5u^2-3u^5-1 , \qquad u \in \mathbb{R}.
\end{align}

We assume that we cannot directly access $\varphi(u)$, but instead have $N$ pairs of samples $(u_i, f^\delta_i)$ for $i \in \{1, \ldots, N\}$ with

\begin{align}
\label{eq:results-fx-additive-noise}
    f^\delta_i = \varphi(u_i)+\varepsilon_i, \;\;\;\; \varepsilon_i\sim\mathcal{N}(0,\sigma^2),
\end{align}

i.e., measured data is affected by additive noise $\varepsilon_i$ sampled from a Gaussian distribution with zero mean and standard deviation $\sigma$, and the superscript $\delta$ refers to the resulting data error $\delta = \| f - f^\delta \|$. We choose $\sigma=0.1$ for our experiments and an example of sampled values is shown in Figure~\ref{fig:poly-1d-res--noisy-images--bottom}.

A polynomial regression model of order $d$ that matches the $N$ sampled data pairs \eqref{eq:results-fx-additive-noise} can be written as a system of linear algebraic equations in the form 
\begin{align}
\label{eq:results-poly-1d-regression-model}
    \Phi w = f^\delta,
\end{align}
where $w\in\mathbb{R}^{(d+1)}$ are the coefficients of the polynomial, $f^\delta \in\mathbb{R}^N$ is the vector containing the noisy samples $f_i^\delta$, and $\Phi\in\mathbb{R}^{N\times(d+1)}$ is the Vandermonde matrix obtained from the values of $x_i^p$ for $p \in \{0, \ldots, d\}$. 

We use $d=75$ and $N=50$ in our experiments.

Traditionally, the goal is finding the sparse set of coefficients $w = (w_0, w_1, \dots, w_d)$ of the $d$-th order polynomial model via the LASSO model
\begin{align}
    w_\alpha \in \argmin_{w \in \R^{1 + d}} \left\{ \frac12 \| \Phi w - f^\delta \|^2 + \alpha \| w \|_1 \right\} ,\label{eq:lasso}
\end{align}
where $\| w \|_1 = \sum_{j = 0}^d | w_j |$ is the one-norm of the coefficient vector $w$.

Now instead of finding $w_\alpha$, our goal is to find data $g_\alpha$ such that the true coefficients $w^\dag$ are a solution of \eqref{eq:lasso} for data $g_\alpha$ (instead of $f^\delta$), respectively the source condition element $v$ using $\Phi^\top v \in \partial \lVert w^\dagger \rVert_1$, where the subdifferential of the one-norm is defined (component-wise) as
\begin{align*}
    (\partial \lVert w^\dagger \rVert_1)_j := 
    \begin{cases} 
     {1}    &  \text{if } w_j   >  0\\ 
     [-1,1] &  \text{if } w_j   =  0\\ 
     {-1}   &  \text{if } w_j   <  0
    \end{cases} .
\end{align*}

For our example, the coefficients $w^{\dagger}$ are zero except for $w^\dagger_0 = -1$, $w^\dagger_2 = 5$ and $w^\dagger_5 = -3$. To compute the source condition element $v$, we minimise $G_J(v)$ iteratively as shown in \eqref{eq:graddesc}, with \begin{align}
    G_{\lVert \cdot \rVert_1} (v) = B_{\| \cdot \|_1}(w^\dagger, w^\dagger + \Phi^\top v) .\label{eq:lasso-poly-1d}
\end{align}
Following \eqref{eq:gradient}, we observe 
$$\nabla G_{\lVert \cdot \rVert_1}(v) = \Phi\left( \text{prox}_{\| \cdot \|_1}\left( w^{\dagger} + \Phi^\top v \right) - w^{\dagger} \right) .$$
In this particular instance, the proximal map $\text{prox}_{\| \cdot \|_1}$ has the closed form solution 
\[\text{prox}_{\| \cdot \|_1}=\text{sign}(\cdot)\,\text{max}(\cdot - 1, 0) \: ,\]
which is also known as the soft thresholding or soft shrinkage operator. 

In order to compute the source condition element via \eqref{eq:graddesc}, we initialise $v^0=0$, and then we iteratively compute
$v^{k+1} = v^{k} - \tau \nabla G_{\lVert \cdot \rVert_1}(v^k)$, with step-size $\tau=1/\lVert \Phi \rVert^2$. Further, we accelerate this procedure by using a Nesterov accelerated version \cite{nesterov1983method} as described in \cite{benning2021bregman}. We stop iterating when either $2 \cdot 10^{8}$ iterations have passed, or when the Euclidean norm of $\nabla G_J$ is smaller than $10^{-12}$.
The norm of the computed source condition element is $\lVert v \rVert \approx 15.32$. The small norm of the source condition element indicates that we can retrieve the weights $w$ reasonably well with the LASSO model even in the presence of noise, as the worst-case error in the estimate \eqref{eq:error-estimate-linear} is not amplified too strongly.
To check that we have really computed a source condition element $v$, we need to compare $\Phi^\top v$ and $\text{sign}(w^\dagger)$. 
This comparison, together with the source condition element, is shown in the left column of Figure~\ref{fig:poly-1d-res-comparison--bottom}.

Finally, we compute the range condition $g_\alpha = \alpha v + \Phi w^{\dagger}$, where $\alpha = \delta/ \lVert v \rVert$, and we show the result in the left column of Figure~\ref{fig:poly-1d-res-RC-test--bottom}. Remember that the data $g_\alpha$ is the data that, when fed into the LASSO operator \eqref{eq:lasso} (instead of $f^\delta$), ensures that the LASSO operator returns the coefficients $u^\dagger$.\\

\noindent The same experiment is now repeated with the higher-order polynomial
\begin{align}
\label{eq:results-fx-1d-polynomial-regression-deg20}
    \varphi(u) = 5u^2-3u^5-\frac{3}{2}u^{13}+\frac{1}{2}u^{20}-1, \qquad u \in \mathbb{R},
\end{align}
and the same Vandermonde matrix $\Phi \in \mathbb{R}^{N \times (d + 1)}$ as before. As outlined in the introduction to this section, this second example is chosen because the problem of identifying coefficients of a polynomial of higher degree is more ill-conditioned compared to the previous example, and $\ell^1$-regularisation has a much more difficult task at hand. We therefore assume that either the source condition will be violated for this example, or there exists a source condition element, but with much larger norm compared to the previous example.

In what follows, $w^\dagger$ must be changed accordingly to include the additional coefficients of this polynomial. Following the previous steps, we compute the source condition element $v$ from the same initial (zero) vector, using the same iterative algorithm. 
Once the iteration is completed, we retrieve a source condition element with norm $\lVert v \rVert \approx 11359$. Convergence is obtained after approximately $32.66\cdot 10^6$ iterations, which is much more than the iterations required for the simpler model to converge (which took less than $10^4$ iterations). This indicates that the new problem is indeed harder to solve, as it takes more iterations and more time to converge to the solution. However, as the algorithm has converged, we expect that $v$ satisfies the source condition, i.e., $\Phi^\top v \in \partial w^\dagger $. This comparison, together with the source condition $v$, is shown in the right column of Figure~\ref{fig:poly-1d-res-comparison--bottom}. We can see that for both models we are able to retrieve the source condition as the two solutions both verify \eqref{eq:sc}, but we can see larger oscillations and a larger norm in the $20$th order polynomial, which indicates that the problem is harder than the $5$th order polynomial because of ill-conditioning.

Using the inequality \eqref{eq:error-estimate-linear}, we can conclude that we get a better bound for the fifth-order polynomial since the noise level $\delta=\lVert y-y^\delta \rVert$ is the same for both models. Similar to the previous example, we also visualise the corresponding range condition data $g_\alpha$ in Figure \ref{fig:poly-1d-res-RC-test--bottom}.

\begin{figure}
\label{fig:poly-1d-res--noisy-images}
\subfloat{\includegraphics[width=0.45\textwidth]{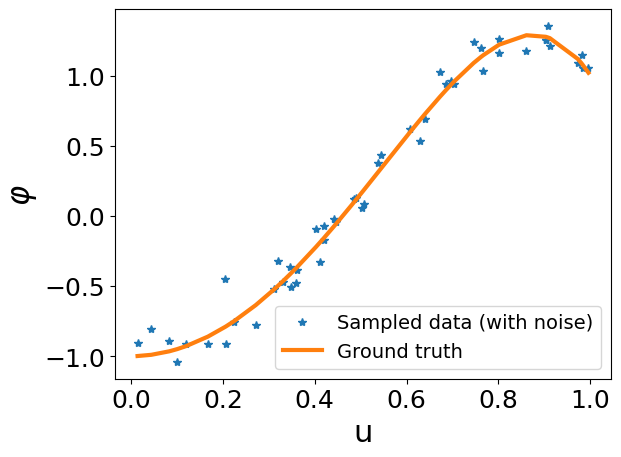}\label{subfig:poly-1d-res--noisy-images-order5}}
\hspace{0.05\textwidth}
\subfloat{\includegraphics[width=0.45\textwidth]{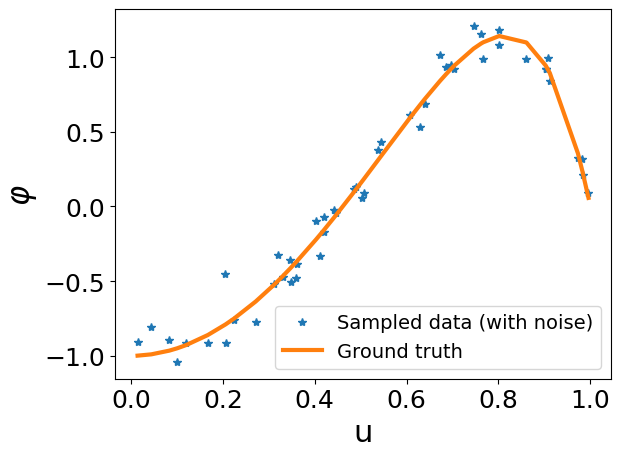}\label{subfig:poly-1d-res--noisy-images-order20}}
   \caption{Noisy sampled data \eqref{eq:results-fx-additive-noise}, for the ground truth $\varphi$ defined  by in \eqref{eq:results-fx-1d-polynomial-regression_deg5} (left) and \eqref{eq:results-fx-1d-polynomial-regression-deg20} (right).}
   \label{fig:poly-1d-res--noisy-images--bottom}
\end{figure}

\begin{figure}
\label{fig:poly-1d-res-comparison}
\subfloat{\includegraphics[width=0.45\textwidth]{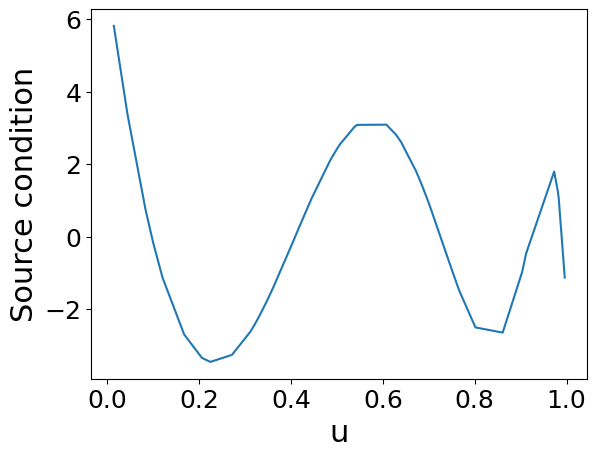}\label{subfig:poly-1d-res-SC-order5}}
\hspace{0.05\textwidth}
\subfloat{\includegraphics[width=0.45\textwidth]{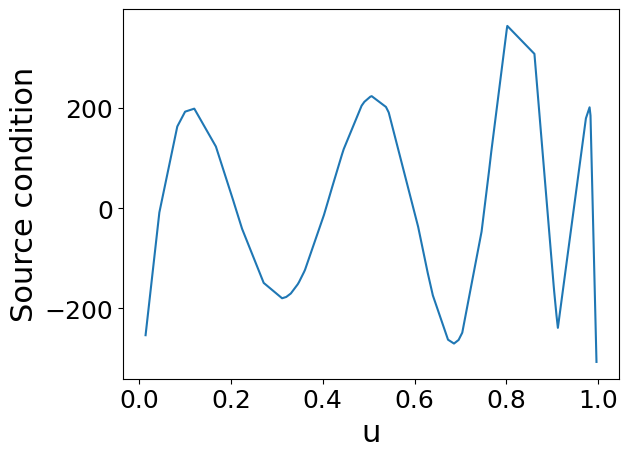}\label{subfig:poly-1d-res-SC-order20}}\\
\subfloat{\includegraphics[width=0.45\textwidth]{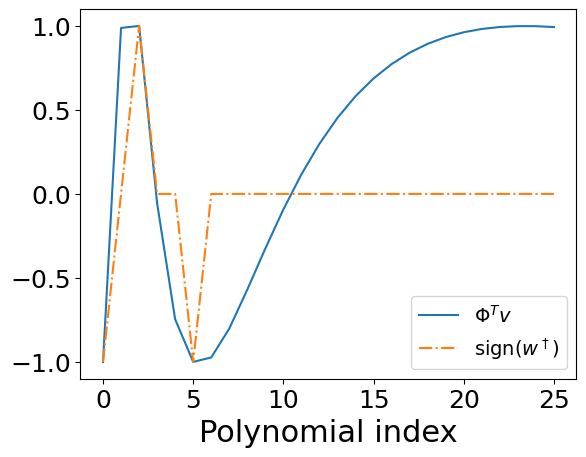}\label{subfig:poly-1d-res-comparison-order5}}
\hspace{0.05\textwidth}
    \subfloat{\includegraphics[width=0.45\textwidth]{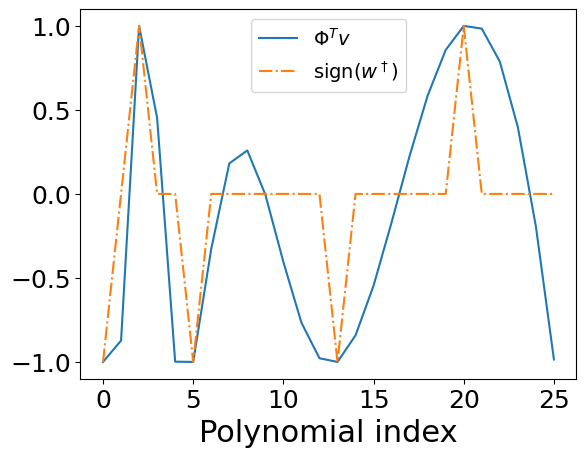}\label{subfig:poly-1d-res-comparison-order20}}
   \caption{
   Top row: source condition element $v$ computed for  examples \eqref{eq:results-fx-1d-polynomial-regression_deg5} (left) and  \eqref{eq:results-fx-1d-polynomial-regression-deg20}  (right). Note the difference in the  ordinate scales: the norm of the source condition element for the higher order polynomial \eqref{eq:results-fx-1d-polynomial-regression-deg20} is much higher it and has larger oscillations than the ones for the lower order polynomial \eqref{eq:results-fx-1d-polynomial-regression_deg5}. Bottom row: comparison between $\Phi^\top v$ and the sign of the true coefficients $w^\dagger$. For both plots, every time $\text{sign}(w^\dagger)=\pm 1$, then also $\Phi^\top v = \pm 1$. This indicates that the estimated source condition $v$ is correct.
   }
   \label{fig:poly-1d-res-comparison--bottom}
\end{figure}

\begin{figure}
\label{fig:poly-1d-res-RC-test}
\subfloat{\includegraphics[width=0.45\textwidth]{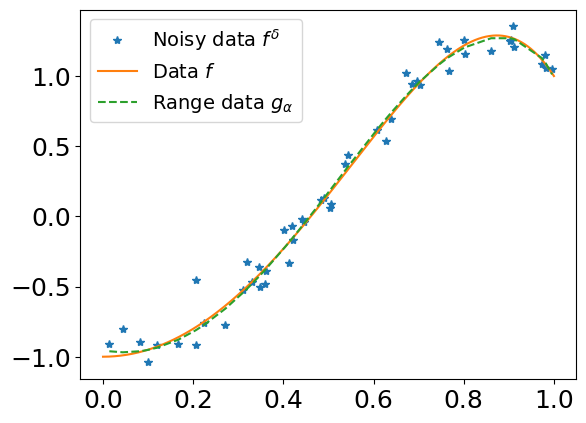}\label{subfig:poly-1d-res-RC-test-order5}}
\hspace{0.05\textwidth}
\subfloat{\includegraphics[width=0.45\textwidth]{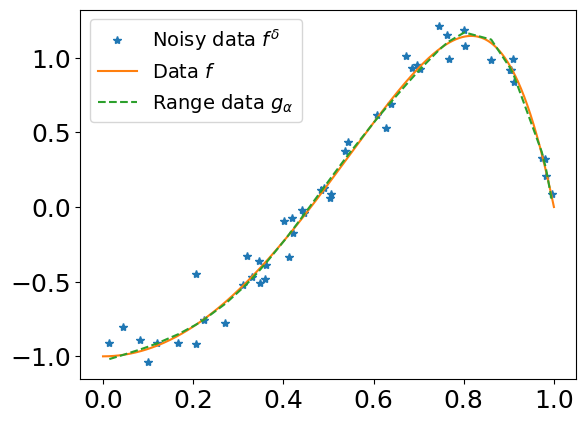}\label{subfig:poly-1d-res-RC-test-order20}}
   \caption{
   Comparison between noisy samples $f^\delta$, data $f$ and range condition element $g_\alpha$ for $\alpha = \delta / \| v \|$, computed for  examples \eqref{eq:results-fx-1d-polynomial-regression_deg5} (left) and  \eqref{eq:results-fx-1d-polynomial-regression-deg20}  (right).
   }
   \label{fig:poly-1d-res-RC-test--bottom}
\end{figure}

\begin{figure}[t]
\subfloat[$v^K$]{\includegraphics[width=0.45\textwidth]{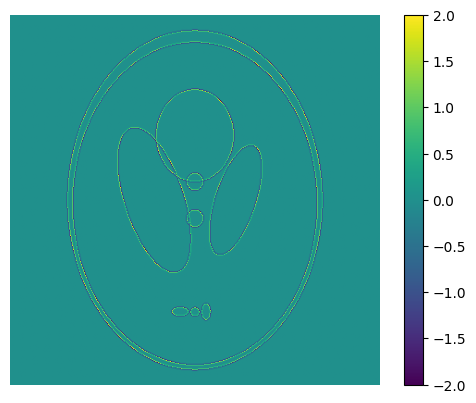}\label{subfig:slp_full_sc}}
\hspace{0.05\textwidth}
\subfloat[$\sqrt{|q_1^K|^2 + |q_2^K|^2}$]{\includegraphics[width=0.45\textwidth]{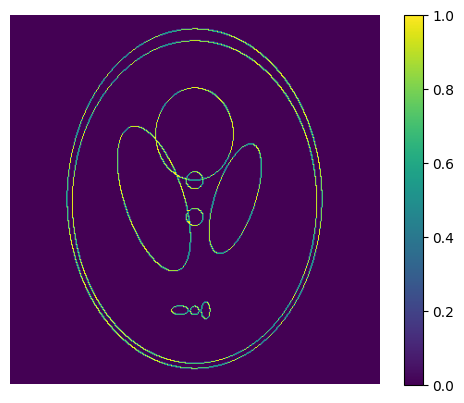}\label{subfig:slp_full_sg}}
\caption{Figure \ref{subfig:slp_full_sc} depicts the source condition element $v^K$ that satisfies $v^K \in \partial \text{TV}(u^\dagger)$ for the Shepp-Logan phantom $u^\dagger$. Figure \ref{subfig:slp_full_sg} shows the corresponding vector $q^K$ such that $A^\top q^K = v^K$. As one would expect  this for an element in the subgradient, the Euclidean norm of $q^K$ with respect to the vector components is bounded by one.}\label{fig:slp-sc-verification}
\end{figure}

\subsection{2D example: Fourier sub-sampling}\label{sec:fourier-subsampling} \label{subsec:fourier}
As a next example, we consider the inverse problem of sub-sampling the Fourier domain of two-dimensional images. If we restrict ourselves to Cartesian grids, we can describe this inverse problem mathematically via the operator equation
\begin{align*}
    S\mathcal{F}u^\dagger = f ,
\end{align*}
where $u^\dagger \in \mathbb{R}^{n_y \times n_x}$ denotes the unknown, two-dimensional discrete image, $f \in \mathbb{C}^m$ the sub-sampled Fourier data, $\mathcal{F}\colon\mathbb{R}^{n_y \times n_x} \rightarrow \mathbb{C}^{n_y \times n_x}$ the two-dimensional discrete Fourier transform, i.e., 
\begin{align*}
    (\mathcal{F}u)_{pq} = \frac{1}{\sqrt{n_x n_y}} \sum_{l = 0}^{n_y - 1} \sum_{j = 0}^{n_x - 1} u_{lj} \, e^{- i \frac{2\pi p l}{n_y}} e^{- i \frac{2\pi q j}{n_x}} ,
\end{align*}
for $p \in \{0, \ldots, n_y - 1\}$ and $q \in \{0, \ldots, n_x - 1\}$, and $S \colon \mathbb{C}^{n_y \times n_x} \rightarrow \mathbb{C}^m$ the sampling operator that selects samples from the Cartesian grid. Since the operator $\mathcal{F}$ is orthogonal, it holds that $\mathcal{F}^\top=\mathcal{F}^{-1}$. For our numerical experiments we choose $J$ to be the (discretised) isotropic total variation, i.e., $J = \text{TV} \colon\mathbb{R}^{n_y \times n_x} \rightarrow \mathbb{R}$ with
\begin{align*}
    \text{TV}(u) = \sum_{i = 1}^{n_y - 1} \sum_{j = 1}^{n_x - 1} \sqrt{ \left| u_{(i + 1) j} - u_{i j} \right|^2 + \left| u_{i (j + 1)} - u_{i j} \right|^2 } .
\end{align*}
In order to compute source condition elements $\mathcal{F}^{-1} S^\top v \in \partial \text{TV}(u^\dagger)$, we make use of the range condition reformulation described in Section \ref{sec:rc-condition} and minimise \eqref{eq:rc-minimisation} via explicit coordinate descent as described in \eqref{eq:coord_desc}, which for our problem reads
\begin{align}
\begin{split}
    v^{k + 1} &= v^k - \tau S \mathcal{F} \left( \mathcal{F}^{-1} S^\top v^k - A^\top q^k \right) ,  \\
    q^{k + 1} &= q^k - \sigma \left( A \left( A^\top q^k - \mathcal{F}^{-1} S^\top v^{k + 1} \right) + \text{prox}_{\| \cdot \|_{2, 1}}\left( Au^\dagger + q^k \right) - Au^\dagger \right) .
\end{split}\label{eq:coordinate-descent}    
\end{align}
\begin{figure}[h]
\subfloat[Fourier transformed data]{\includegraphics[width=0.45\textwidth]{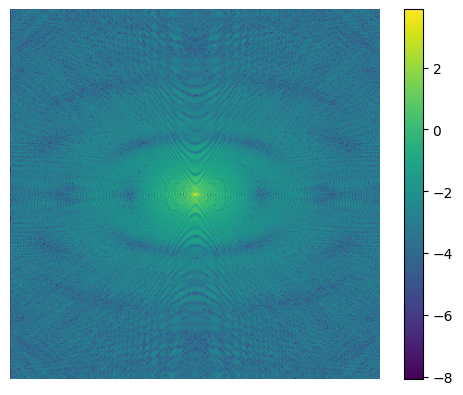}}
\hspace{0.05\textwidth}
\subfloat[Sub-sampled Fourier transformed data]{\includegraphics[width=0.45\textwidth]{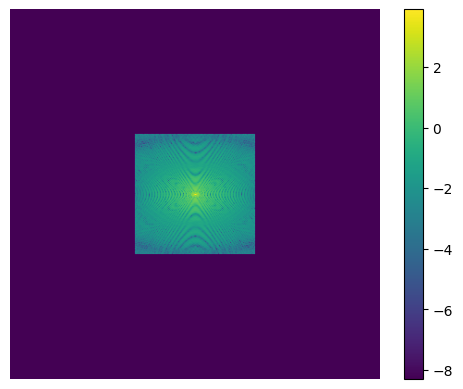}}
\caption{The Fourier transform of the Shepp-Logan phantom depicted in Figure \ref{subfig:slp-sanity1}, and the sub-sampled Fourier transform that emulates a simple low-pass filter. Note that we depict the logarithm of the absolute value of the Fourier transformed data plus the constant $1/4000$ for better visualisation.}\label{fig:slp-fourier-images}
\end{figure}%
Here $A\colon\mathbb{R}^{n_y \times n_x} \rightarrow \mathbb{R}^{(n_y - 1) \times (n_x - 1) \times 2}$ is the (forward) finite-difference discretisation of the gradient, i.e.,
\begin{align*}
    (A u)_{ijp} = \begin{cases}
     u_{(i + 1)j} - u_{ij} & \text{for } p = 1 , \\
     u_{i(j + 1)} - u_{ij} & \text{for } p = 2 ,
    \end{cases}  ,
\end{align*}
for $i \in \{1, \ldots , n_y - 1\}$ and $j \in \{1, \ldots, n_x - 1\}$, and $\text{prox}_{\| \cdot \|_{2, 1}}$ denotes the proximal map with respect to the function $\| \cdot \|_{2, 1} \colon \mathbb{R}^{(n_y - 1) \times (n_x - 1) \times 2} \rightarrow \mathbb{R}$ defined as
\begin{align*}
    \| q \|_{2, 1} = \sum_{i = 1}^{n_y - 1} \sum_{j = 1}^{n_x - 1} \sqrt{ \left| q_{ij1} \right|^2 + \left| q_{ij2} \right|^2 } .
\end{align*}
The proximal map for this function reads
\begin{align*}
    \left( \text{prox}_{\| \cdot \|_{2, 1}}(z) \right)_{ijp} = \frac{z_{ijp}}{\sqrt{|z_{ij1}|^2 + |z_{ij2}|^2}}\max\left(\sqrt{|z_{ij1}|^2 + |z_{ij2}|^2} - 1, 0 \right) ,
\end{align*}
for $i \in \{1, \ldots, n_y - 1\}$, $j \in \{1, \ldots, n_x - 1\}$ and $p \in \{1, 2\}$. We choose the positive step-size parameters $\tau = 1$ and $\sigma = 1/8$, in order to guarantee that \eqref{eq:coordinate-descent} converges to a global minimiser of \eqref{eq:rc-minimisation} for any initialisation, assuming that one exists. Note that throughout this section, we will always initialise the variables $v^0$ and $q^0$ with zeros of the correct dimensions.

In the following, we show the results for  two choices of $u^\dagger$:  the famous Shepp-Logan phantom depicted in Figure~\ref{subfig:slp-sanity1},  and (a grayscale version of) the image of astronaut Eileen Collins depicted in Figure~\ref{subfig:ec-sanity1}. Intuitively, we expect the Shepp-Logan phantom to satisfy the source condition with reasonably low norm of the corresponding source condition element, since  it is piecewise constant and therefore well suited for total variation regularisation. The astronaut image of Eileen Collins, on the other hand, is not piecewise constant but textured and contains both jumps and gradual increases in intensity and therefore ill suited for total variation regularisation, so that we expect that a source condition is either not satisfied, or that, if it is satisfied in this discrete setting, it is only satisfied with a large norm for the corresponding source condition element.

\subsubsection{Shepp-Logan phantom}\label{sec:fourier-subsampling-shepp-logan}
We begin our discussion of numerical results in this section with the Shepp-Logan phantom. We a use a gray-scale version with $n_y = n_x = 400$ pixels. Before we begin solving the inverse problem of recovering the phantom from low-pass filtered Fourier data, we verify empirically that the Shepp-Logan phantom $u^\dagger$ satisfies the source condition $v \in \partial \text{TV}(u^\dagger)$.. In order to compute $v$, we evaluate \eqref{eq:coordinate-descent} with $S$ being the identity operator, without any Fourier operator and with the parameter choices $\tau = 1$ and $\sigma = 1/9$ since $\| A \| < 8$ (cf. \cite{chambolle2004algorithm}), and stop the iteration once the Euclidean norm of the partial derivatives satisfies $\frac12 ( \| A(A^\top q^K - v^K) + \text{prox}_{\| \cdot \|_{2, 1}}(Au^\dagger + q^K) - Au^\dagger \| + \| v^K - A^\top q^K \| ) \leq 3.84 \times 10^{-14}$, which is chosen because it is very close to machine accuracy and therefore effectively zero. The corresponding iterates $v^K$ and $q^K$ are visualised in Figure \ref{fig:slp-sc-verification}. The Euclidean norm of $v^K$ is approximately $\| v^K \| \approx 101.78$, which means that we can accurately quantify error estimates of the form \eqref{eq:error-estimate-linear} for the Shepp-Logan phantom for the inverse problem of denoising.

We now want to move on to the inverse problem of estimating the Shepp-Logan phantom from low-pass filtered Fourier data. We choose a square low-pass filter of size $130 \times 130$ around the centre frequency. The Fourier transform of the Shepp-Logan phantom and the corresponding low-pass filter are visualised in Figure \ref{fig:slp-fourier-images}.

We now employ the same algorithmic strategy, namely \eqref{eq:coordinate-descent}, but where $S$ corresponds to the low-pass sub-sampling. In contrast to the previous example, the decrease of the norm of the partial derivatives is much slower, and we stop the iteration after 1000 iterations, with an approximate value of $0.12$. This indicates that the problem is computationally much harder to solve or that a source condition element does not exist and can therefore not be computed. However, we can certainly use the output $v^K$ and $q^K$ for $K = 1000$ as an approximate source condition element, which we visualise in Figure \ref{fig:slp-sc}. 

\begin{figure}[t]
\centering
\subfloat[$v^K$]{\includegraphics[height=3.84cm]{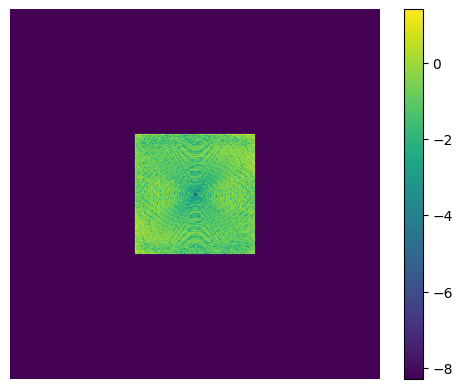}}
\hspace{0.04\textwidth}
\subfloat[Backprojection $\mathcal{F}^{-1} S^\top v^K$]{\includegraphics[height=3.84cm]{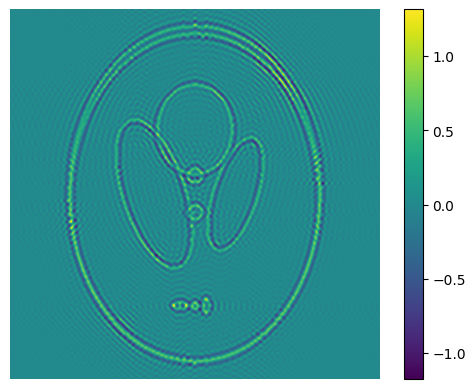}\label{subfig:slp-backproj}}
\hspace{0.04\textwidth}
\subfloat[$\sqrt{|q_1^K|^2 + |q_2^K|^2}$]{\includegraphics[height=3.84cm]{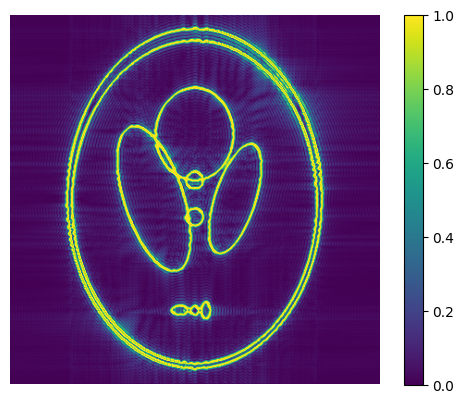}\label{subfig:slp-sg}}
\caption{The approximate source condition element $v^K$ computed with \eqref{eq:coordinate-descent} after $K = 1000$ iterations (left), the application of $\mathcal{F}^{-1}S^\top$ to $v^K$ (middle) and the Euclidean vector norm of the vector-field $q^K$ of the corresponding subgradient $A^\top q^K$ (right).}\label{fig:slp-sc}
\end{figure}

We immediately observe that the scaling of $\mathcal{F}^{-1} S^\top v^K$ is fairly different from $v^K$ in the denoising case, with values mostly in the range of $[-1, 1]$ instead of $[-2, 2]$. Interestingly, the norm of $v^K$ after $K= 1000$ iterations for this low-pass filter example is $\| v^K \| \approx 72.79$, which is significantly smaller than the $101.78$ that we encountered in the denoising case. A reason for this could be that we did not  solve the optimisation problem to the necessary accuracy . Another reason could be that the small $130 \times 130$ window of the low-pass filter, which allows $v^k$ to only have 16900 instead of 160000 non-zero values, has an impact on the norm of the (approximate) source condition element. The latter also explains the difference in scale of the projection $\mathcal{F}^{-1}S^\top v^K$ shown in Figure \ref{subfig:slp-backproj}. We also want to emphasise that the Euclidean norm of the vector-field $q^K$ is not strictly bounded by one (even though the visualisation in Figure \ref{subfig:slp-sg} seems to suggest this), but that some values exceed this threshold. This is  proof that for this example a source condition element is approximated, but not found.

\begin{figure}[t]
\subfloat[range condition data $g_\alpha^K$]{\includegraphics[width=0.425\textwidth]{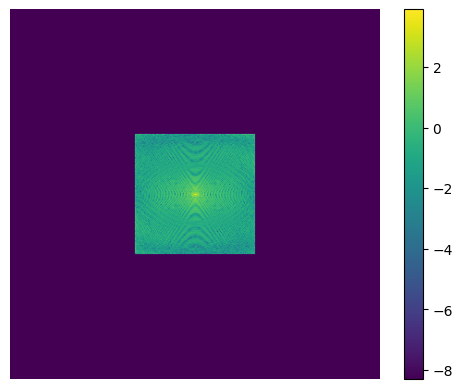}\label{subfig:slp-rc}}
\hspace{0.025\textwidth}
\subfloat[ground truth $u^\dagger$]{\includegraphics[width=0.45\textwidth]{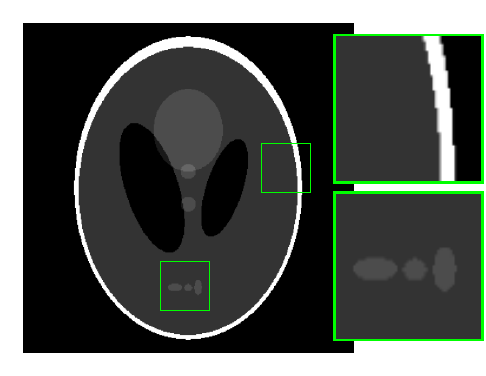}\label{subfig:slp-sanity1}}\\
\subfloat[approximate solution $u^N$]{\includegraphics[width=0.45\textwidth]{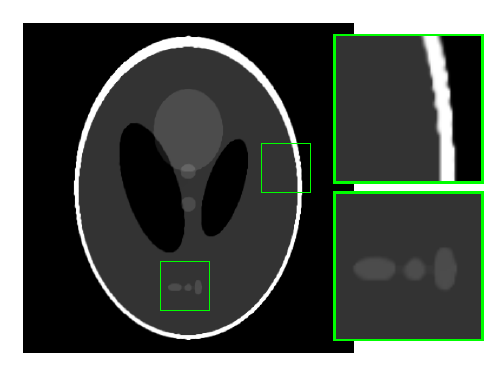} \label{subfig:slp-sanity2}}
\subfloat[low-pass filtered reconstruction $\mathcal{F}^{-1}S^\top S\mathcal{F}u^\dagger$]{\includegraphics[width=0.45\textwidth]{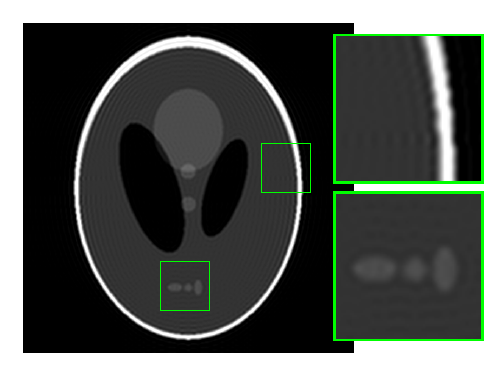} \label{subfig:slp-sanity3}}
\caption{This figure shows the (approximate) range condition data $g_\alpha^K = \alpha v^K + S\mathcal{F}u^\dagger$ for $\alpha = 1/2$ in Figure \ref{subfig:slp-rc}. In Figure \ref{subfig:slp-sanity1} we see the original Shepp-Logan phantom $u^\dagger$. In Figure \ref{subfig:slp-sanity2} the approximate solution $u^N$ of \eqref{eq:variational-regularisation} computed with the PDHG method for $\alpha = 1/2$ and input data $g_\alpha^K$. In Figure \ref{subfig:slp-sanity3} the linear low-pass filtered reconstruction $\mathcal{F}^{-1}S^\top S\mathcal{F}u^\dagger$. For the last three figures, a close-up of two region of interest is provided (green squares).}\label{fig:slp-rc}
\end{figure}

Another proof that $v^K$ is only an approximate source condition element can be found with the help of the relation between source and range condition \eqref{eq:range-condition}. With the range condition we can immediately characterise data $g_\alpha = \alpha v + Ku^\dagger$ given a source condition element $v$, or $g_\alpha^K = \alpha v^K + S\mathcal{F}u^\dagger$ for our example. We set $\alpha$ to the arbitrary value of $\alpha = 1/2$ and compute a solution of \eqref{eq:variational-regularisation} from the data $g_\alpha^K$ with an implementation of the primal-dual hybrid gradient (PDHG) method \cite{chambolle2016introduction}. We choose the PDHG version described in \cite{benning2021bregman} with step-sizes $\tau = 1/8$ and $\sigma = 1$, and iterate for $N = 1000$ iterations  after which the iterates of the primal ($u^N$) and dual variable ($q^N$) of the PDHG method satisfy $1/2 * (\| u^N - u^{N - 1} \| / \|u^N\| + \| q^N - q^{N - 1} \| / \| q^K \| ) < 6.85 \times 10^{-5}$. The number of iterations have been capped at 1000 because improvements in reconstruction beyond 1000 iterations is marginal. The results together with the data $g_\alpha$ are visualised in Figure \ref{fig:slp-rc}. Please note that while it is straight-forward to verify that an element computed with either \eqref{eq:graddesc} or \eqref{eq:coord_desc} satisfies the source condition (as we have done in Section \ref{subsec:1Dpoly} and also for the fully-sampled example in this section), it is not equally straight-forward to disprove existence of a source condition element when the result of either algorithm does not satisfy the source condition. This could also be a result of very slow convergence, and faster optimisation algorithms with guaranteed convergence after a finite number of iterations may be required.

We see from the close-up in Figure \ref{fig:slp-rc} that the reconstruction does not exactly match the Shepp-Logan phantom, which is what one would expect if $v^K$ was a source condition element.  However, we nevertheless observe that the reconstruction $u^N$ for data $g_\alpha^K$ is a good approximation of the Shepp-Logan phantom and reasonably better than the traditional low-pass filter as one would expect. Another interesting observation is that if we were to define $u^\dagger = u^N$, we can guarantee that $u^\dagger$ satisfies the source condition with source condition element $v$ that satisfies $\| v \| \approx 72.79$. This is lower than the value of $101.78$ that we obtained in the denoising case for the Shepp-Logan phantom, and likely a result of this new $u^\dagger$ being smoother than the Shepp-Logan phantom. However, the nature of the inverse problem with the low-pass filter forward operator might also play a role for the lower value of the norm as it is plausible that for this type of filter errors are amplified less strongly, since the transpose operation $\mathcal{F}^{-1} S^\top$ filters high-frequency errors quite effectively.

We are going to see in Section \ref{sec:optimal-sampling} that we can find a more data-adaptive sampling strategy (compared to the low-pass filter) for which we can recover $u_\alpha \in \mathcal{R}_\alpha(g_\alpha)$ almost perfectly from a smaller number of samples.

\begin{figure}[t]
\subfloat[Fourier transformed data]{\includegraphics[width=0.45\textwidth]{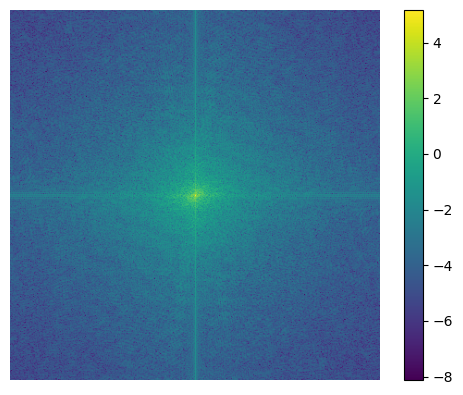}\label{subfig:ec_ft}}
\hspace{0.05\textwidth}
\subfloat[Sub-sampled Fourier transformed data]{\includegraphics[width=0.45\textwidth]{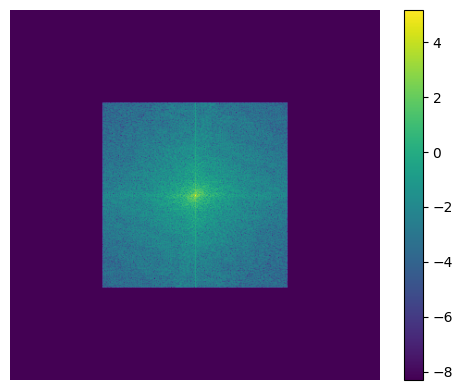}\label{subfig:ec_sub_ft}}
\caption{The Fourier transform of the astronaut image depicted in Figure \ref{subfig:ec-sanity1}, and the sub-sampled Fourier transform that emulates a simple low-pass filter. Note that we depict the logarithm of the absolute value of the Fourier transformed data plus the constant $1/4000$ for better visualisation.}
\end{figure}

\subsubsection{Eileen Collins}\label{sec:fourier-subsampling-astronaut}
We perform identical experiments as in the previous section, but this time we choose an element $u^\dagger$ for which satisfying the source condition is highly unlikely: an image with textures and fine-scale details. We pick the astronaut image of Eileen Collins depicted in Figure \ref{subfig:ec-sanity1}. We use a gray-scale version that is down-scaled to $n_y = 400$ and $n_x = 400$ pixels as our image $u^\dagger$, and we try to empirically verify if the astronaut image satisfies a source condition of the form $\mathcal{F}^{-1} S^\top v \in \partial \text{TV}(u^\dagger)$. 

Similar to the previous example, we choose a square low-pass filter, but this time of size $200 \times 200$ around the centre frequency. The Fourier transform of the Eileen Collins image and the corresponding low-pass filter are visualised in Figures \ref{subfig:ec_ft} and  \ref{subfig:ec_sub_ft}, respectively.

In analogy to the previous section, we evaluate \eqref{eq:coordinate-descent} with the parameter choices $\tau = 1$ and $\sigma = 1/8$, initialise with zero arrays and also stop the iteration after $K = 1000$ iterations when the Euclidean norm of the partial derivatives satisfies $\frac12 ( \| A(A^\top q^K - \mathcal{F}^{-1}S^\top v^K + \text{prox}_{\| \cdot \|_{2, 1}}(Au^\dagger + q^K) - Au^\dagger \| + \| \mathcal{F}^{-1} S^\top v^K - A^\top q^K \|) \leq 0.59$. The corresponding iterates $v^K$ and $q^K$ are visualised in Figures \ref{subfig:ec-sc} and  \ref{subfig:ec-sg}, respectively. We want to emphasise that the Euclidean norm of the vector-field $q^K$ is not strictly bounded by one (even though the visualisation in Figure \ref{subfig:ec-sg} seems to suggest this), but that some values exceed this threshold, which is proof that for this example a source condition element is only approximated, but not found. 

In addition, the Euclidean norm of $v^K$ is approximately $\| v^K \| \approx 255.15$ and much larger compared to the Shepp-Logan example, which is what we would expect from a textured image $u^\dagger$ with fine details.

\begin{figure}[t]
\centering
\subfloat[$v^K$]{\includegraphics[height=3.84cm]{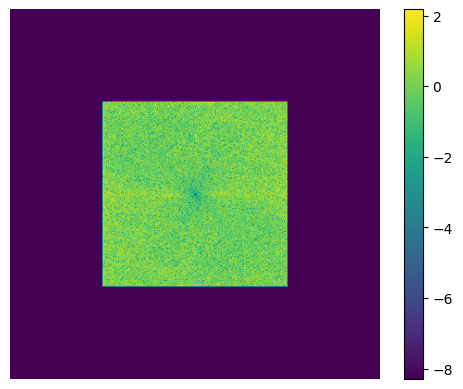}\label{subfig:ec-sc}}
\hspace{0.04\textwidth}
\subfloat[Backprojection $\mathcal{F}^{-1} S^\top v^K$]{\includegraphics[height=3.84cm]{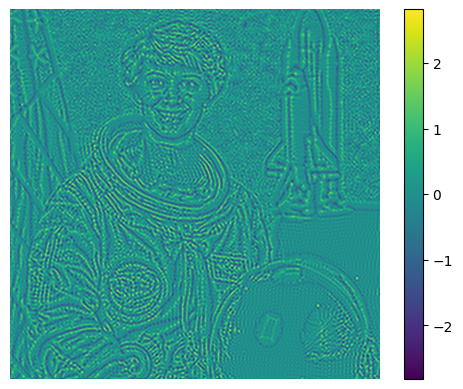}\label{subfig:ec-backproj}}
\hspace{0.04\textwidth}
\subfloat[$\sqrt{|q_1^K|^2 + |q_2^K|^2}$]{\includegraphics[height=3.84cm]{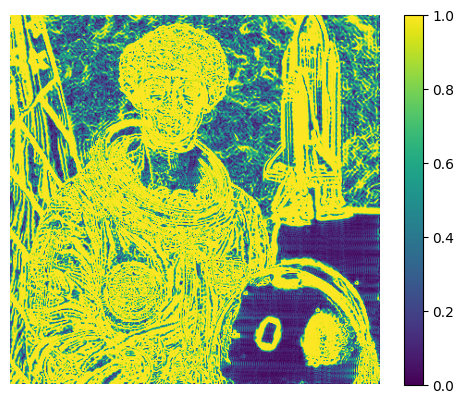}\label{subfig:ec-sg}}
\caption{The approximate source condition element $v^K$ computed with \eqref{eq:coordinate-descent} after $K = 1000$ iterations (left), the application of $\mathcal{F}^{-1}S^\top$ to $v^K$ (middle) and the Euclidean vector norm of the vector-field $q^K$ of the corresponding subgradient $A^\top q^K$ (right).}\label{fig:ec-sc}
\end{figure}

Similarly to the Shepp-Logan phantom, we check  solutions of \eqref{eq:variational-regularisation} for the range data $g_\alpha = \alpha v + S\mathcal{F} u^\dagger$ for $\alpha = 1/2$. For our example, we define $g_\alpha^K = \alpha v^K + S\mathcal{F} u^\dagger$ and compute a solution of \eqref{eq:variational-regularisation} with the PDHG method as described in the previous section (for identical initialisation and parameter choices). After $N = 1000$ iterations we compute primal ($u^N$) and dual ($q^N$) iterates that satisfy $(\| u^N - u^{N - 1} \|/\|u^N\| + \| q^N - q^{N - 1}\|/\|q^N\|)/2 < 10^{-4}$. The result $u^N$ together with the data $g_\alpha^N$ are visualised in Figure \ref{fig:ec-rc}.

\begin{figure}[H]
\centering
\subfloat[range condition data $g_\alpha^K$]{\includegraphics[width=0.425\textwidth]{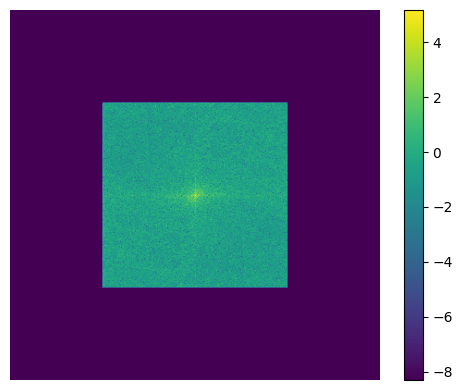}\label{subfig:ec-rc}}
\hspace{0.05\textwidth}
\subfloat[ground truth $u^\dagger$]{\includegraphics[width=0.485\textwidth]{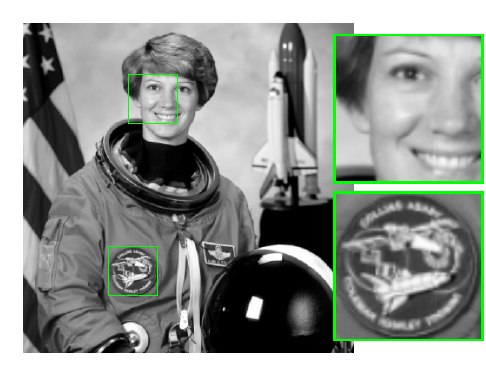}\label{subfig:ec-sanity1}}\\
\subfloat[approximate solution $u^N$]{\includegraphics[width=0.485\textwidth]{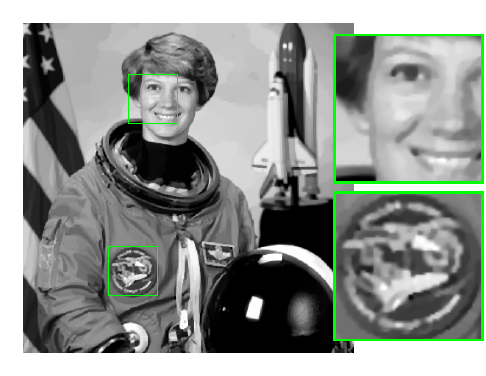}\label{subfig:ec-sanity2}}
\subfloat[low-pass filtered reconstruction $\mathcal{F}^{-1}S^\top S\mathcal{F}u^\dagger$]{\includegraphics[width=0.485\textwidth]{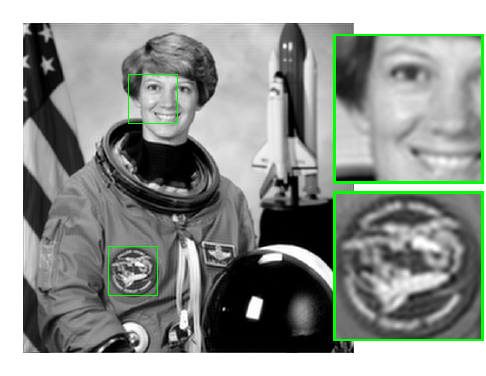}\label{subfig:ec-sanity3}}

\caption{This figure shows the (approximate) range condition data $g_\alpha^K = \alpha v^K + S\mathcal{F}u^\dagger$ for $\alpha = 1/2$ in Figure \ref{subfig:ec-rc}. In Figure \ref{subfig:ec-sanity1} we see the image $u^\dagger$ of Eileen Collins. In Figure \ref{subfig:ec-sanity2} the approximate solution $u^N$ of \eqref{eq:variational-regularisation} computed with the PDHG method for $\alpha = 1/2$ and input data $g_\alpha^K$. In Figure \ref{subfig:ec-sanity3} the linear low-pass filtered reconstruction $\mathcal{F}^{-1}S^\top S\mathcal{F}u^\dagger$. For all three figures, a close-up of two region of interest is provided (green squares).}\label{fig:ec-rc}
\end{figure}

We clearly see that the reconstruction $u^N$ is not identical to $u^\dagger$ but rather a cartoon-like approximation that misses high-scale features such as textures. This does not come as a surprise as we would expect this from a total-variation approximation that does not have access to the high-scale features that the low-pass filter suppresses. We are going to see in Section \ref{sec:optimal-sampling} that we can improve the reconstruction by identifying a more adaptive sampling pattern when estimating the approximate source condition element.

\subsection{Optimal sampling in the Fourier domain}\label{sec:optimal-sampling}
In Section \ref{sec:fourier-subsampling} we studied the empirical computation of source condition and approximate source condition elements for the variational regularisation of the form
\begin{align*}
    u_\alpha \in \argmin_{u \in \mathbb{R}^{n_y \times n_x}} \left\{ \frac12 \left\| S \mathcal{F} u - f^\delta \right\|^2 + \alpha \text{TV}(u) \right\} .
\end{align*}
In this section we want to take a step further and also estimate the sampling pattern that defines the sampling operator $S$. This idea is not new and has applications, for example, in  MRI (cf. \cite{sherry2020learning}). However, in this section we present a much simpler approach for estimating $S$ compared to works such as \cite{sherry2020learning}. Most importantly, the approach can be phrased as a convex optimisation problem depending on how we choose to enforce sparsity of $\tilde{v}$, while most alternative approaches constitute non-convex optimisation problems.

The approach is summarised as follows. We assume that $S$ no longer maps onto $\mathbb{C}^m$, but that $S$ maps onto $\mathbb{C}^{n_y \times n_x}$ and that it is a diagonal operator with zero-entries on the diagonal wherever the sampling mask is zero. If we consider the standard source condition $\mathcal{F}^{-1} S^\top v \in \partial \text{TV}(u^\dagger)$ for the above problem, we can define $\tilde{v} \colon = S^\top v \in \mathbb{C}^{n_y \times n_x}$, and estimate $\tilde{v}$ instead of $v$. Since $\tilde{v}$ has to be sparse by nature in order to emulate sub-sampling, we can estimate $\tilde{v}$ by solving
\begin{align}
    \min_{\tilde{v}, q^\dagger} \left\{ \frac12 \| \mathcal{F}^{-1} \tilde{v} - A^\top q^\dagger \|^2 + G_{\| \cdot \|_{2, 1}}(q^\dagger) + \beta J(\tilde{v}) \right\} ,\label{eq:sparse-sc}
\end{align}
where $J$ is a sparsity-inducing regularisation function, e.g. $J(\tilde{v}) = \| \tilde{v} \|_1 = \sum_{i = 1}^{n_y} \sum_{j = 1}^{n_x} |\tilde{v}_{ij}| $ where $| \cdot |$ denotes the complex modulus, and $\beta > 0$ is a penalisation parameter that controls the level of sparsity of $\tilde{v}$. We then minimise \eqref{eq:sparse-sc} via the proximal alternating linearised minimisation (PALM) algorithm \cite{doi:10.1137/120887795,bolte2014proximal}, i.e., we approximate a solution of \eqref{eq:sparse-sc} by iterating
\begin{align}
\begin{split}
    \tilde{v}^{k + 1} &= \text{prox}_{\beta J}\left( \tilde{v}^k - \tau \left( \tilde{v}^k - \mathcal{F}A^\top q^k \right) \right) , \\
    q^{k + 1} &= q^k - \sigma \left( A\left( A^\top q^k - \mathcal{F}^{-1} \tilde{v}^{k + 1} \right) + \text{prox}_{\| \cdot \|_{2, 1}}\left( q^k + Au^\dagger\right) - Au^\dagger \right)  , 
\end{split}\label{eq:palm}
\end{align}
for $k \in \mathbb{N}$, initial values $\tilde{v}^0$ and $q^0$ and suitable step-size parameters $\tau$ and $\sigma$.

In the following, we approximate $\tilde{v}^K$ and $q^K$ for the arbitrary choice $K = 1000$. We then estimate the mask that determines the zero and non-zero entries on the diagonal of $S$ by simply identifying the zero and non-zero entries of $\tilde{v}$. Please note that we manually enforce that the lowest frequency is included in the mask, as it otherwise would be excluded since total variation subgradients have zero mean. Subsequently, we estimate $v$ by solving \eqref{eq:coordinate-descent} with the estimated sub-sampling operator $S$ from the previous step, for $K = 1000$ iterations. We have conducted experiments for the same examples visualised in Figures \ref{subfig:slp-sanity1} and~\ref{subfig:ec-sanity1}, which are described in the next two sections, and we compare both sampling pattern and reconstructions to those of \cite{sherry2020learning}. For the latter, we use the code provided \href{https://github.com/fsherry/bilevelmri}{here}. Please note that due to the high computational cost and runtime of the method in \cite{sherry2020learning}, we cap the outer iterations of the bilevel optimisation algorithm at 250, and only operate on down-scaled $256 \times 256$ version of the images.

\subsubsection{Shepp-Logan phantom}
We begin with the Shepp-Logan phantom as seen in Figure \ref{subfig:slp-sanity1} and compute the corresponding element $\tilde{v}^K$ via \eqref{eq:palm} with zero-initialisation, $J = \beta \| \cdot \|_1$ and the parameter $\beta = 0.1$. The corresponding $\tilde{v}^K$, the mask of all non-zero coefficients and the mask of the 16900 largest Fourier coefficients (in magnitude) are visualised in Figure \ref{fig:slp-sparse-sc}. Please note that for $\beta = 0.1$ the number of non-zero coefficients of $\tilde{v}^K$ after $K = 1000$ iterations is 14553, which is comparable but even slightly lower than the 16900 non-zero coefficients that were used in the low-pass filter example in Section \ref{sec:fourier-subsampling-shepp-logan}. The number of non-zero coefficients of the approach in \cite{sherry2020learning} after 250 iterations with the same sparsity parameter $\beta$ is 18642, which is much larger and likely a result of the early termination as a consequence of the high computational cost.

\begin{figure}[H]
\centering
\subfloat[$\tilde{v}^K$]{\includegraphics[width=0.3\textheight]{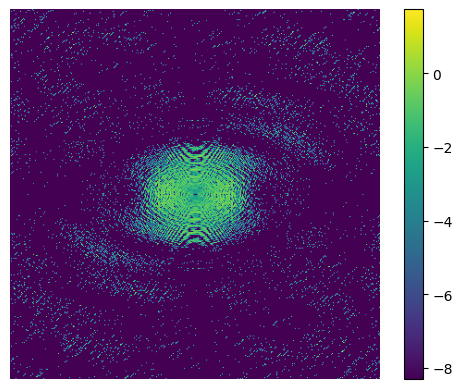}\label{subfig:slp-sparse-sc}}
\hspace{0.03\textwidth}
\subfloat[Sampling pattern]{\includegraphics[width=0.255\textheight]{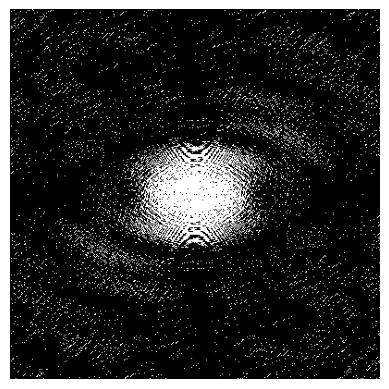}\label{subfig:slp-sparse-mask}} \\
\hspace{0.01\textwidth}
\subfloat[Learned sampling pattern with~\cite{sherry2020learning}]{\includegraphics[width=0.245\textheight]{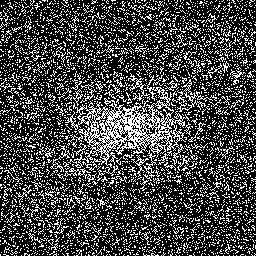} \label{subfig:slp-optimal-mask-comparison}}
\hspace{0.1\textwidth}
\subfloat[Largest Fourier coefficients]{\includegraphics[width=0.255\textheight]{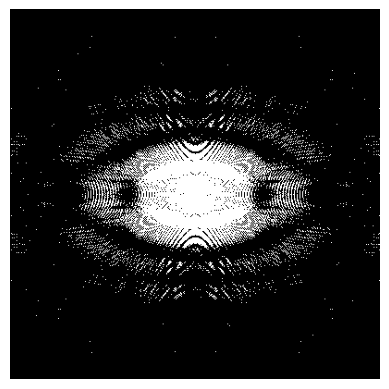} \label{subfig:slp-optimal-mask}}
\caption{Figure \ref{subfig:slp-sparse-sc} shows the (approximate) source condition element for the Shepp-Logan phantom computed with \eqref{eq:palm}. Figure \ref{subfig:slp-sparse-mask} shows the corresponding mask with a 9\% ratio of non-zero to zero coefficients. For comparison, Figure~\ref{subfig:slp-optimal-mask-comparison} shows the mask learned with the approach proposed in~\cite{sherry2020learning} with a 28.45\% ratio and Figure \ref{subfig:slp-optimal-mask} shows the mask of the Fourier coefficients with largest magnitude  with a 10.5\% ratio of non-zero to zero coefficients.}\label{fig:slp-sparse-sc}
\end{figure}

We want to emphasise that the learned mask differs substantially from the mask that stems from the largest Fourier coefficients (in magnitude). In particular, low-frequency information is traded in for high-frequency information since the total variation regularisation is very effective at establishing information from limited low-frequency data but very ineffective at generating high-frequency information. What is a little bit surprising is that the learned sampling pattern with \cite{sherry2020learning} also differs from the learned mask with the proposed approach; it appears to be more random, albeit with a heavy weighting of centre frequencies.

Next, given the mask, we estimate $v$ via \eqref{eq:coordinate-descent} and obtain $v^K$ for $K = 1000$. Note that the corresponding norm is $\| v^K \| \approx 113.93$, which is slightly larger but fairly comparable to the norm that we established in the denoising setting. 
Subsequently, we perform another sanity-check and compute an approximation of \eqref{eq:variational-regularisation} for data $g_\alpha^K = \alpha v^K + S\mathcal{F}u^\dagger$ for $\alpha = 1/2$ via the PDHG method with the same initialisation and parameter configurations as described in Section \ref{sec:fourier-subsampling-shepp-logan} and visualise the results in Figure \ref{fig:slp-sparse-sc}.

\begin{figure}[H]
\centering
\subfloat[Ground truth $u^\dagger$]{
\includegraphics[width=0.33\textwidth]{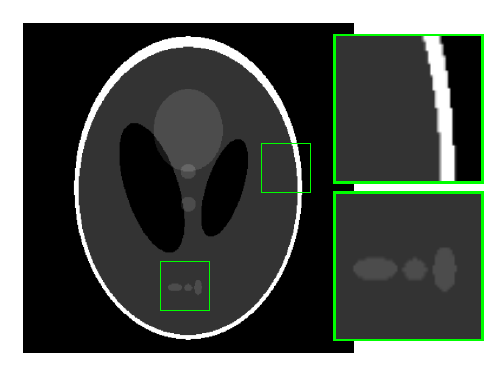}\label{subfig:slp-sparse-sanity1}}
\subfloat[Approximate solution $u^N$]{
\includegraphics[width=0.33\textwidth]{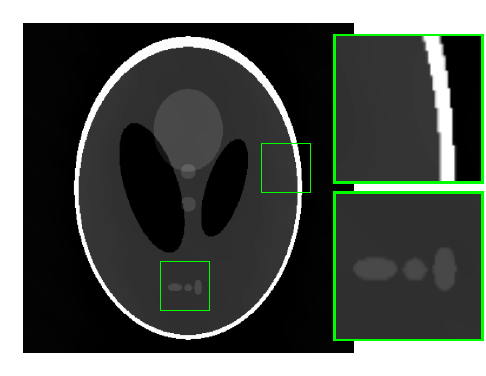}\label{subfig:slp-sparse-sanity2}}
\subfloat[Projection $\mathcal{F}^{-1} S^\top S \mathcal{F} u^\dagger$]{
\includegraphics[width=0.33\textwidth]{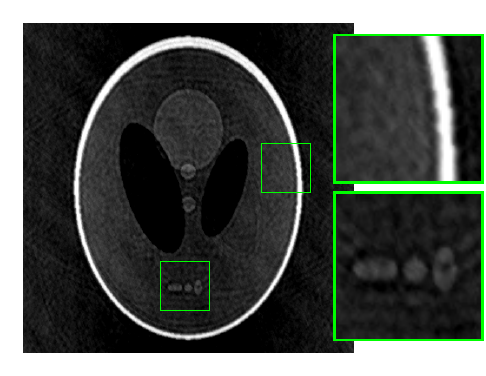}\label{subfig:slp-sparse-sanity3}} \\
\subfloat[Low-pass filtered reconstruction]{
\includegraphics[width=0.33\textwidth]{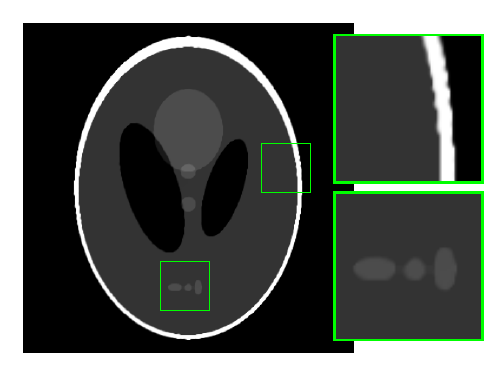}\label{subfig:slp-sparse-sanity4}}
\subfloat[Reconstruction with~\cite{sherry2020learning}]{
\includegraphics[width=0.33\textwidth]{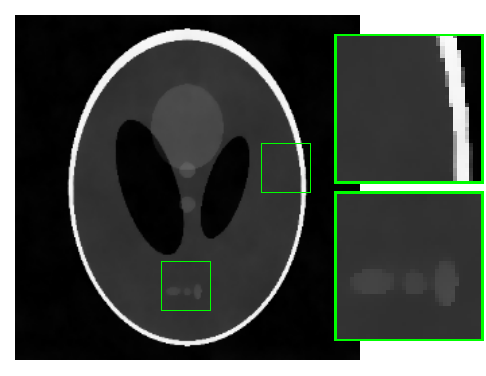}\label{subfig:slp-sparse-sanity5}}
\caption{In Figure~\ref{subfig:slp-sparse-sanity1} we show the
Shepp-Logan phantom ; in Figure~\ref{subfig:slp-sparse-sanity2} the approximation $u^N$ of \eqref{eq:variational-regularisation} with the learned mask and corresponding range data ; in Figure~\ref{subfig:slp-sparse-sanity3} the projection $\mathcal{F}^{-1} S^\top S \mathcal{F} u^\dagger$; in Figure~\ref{subfig:slp-sparse-sanity4} the approximation $u^N$ of \eqref{eq:variational-regularisation} with the low-pass filter from Section \ref{sec:fourier-subsampling-shepp-logan}; in Figure~\ref{subfig:slp-sparse-sanity5} the reconstruction based on~\cite{sherry2020learning}. For all figures, a close-up of two region of interest is provided (green squares).}
\label{fig:slp-sparse-sc}
\end{figure}

In comparison to Section \ref{sec:fourier-subsampling-shepp-logan}, we observe that we recover the Shepp-Logan phantom almost perfectly whilst using a sampling operator that uses fewer samples than the low-pass sampling operator. Please note that the reconstruction also outperforms the reconstruction with \cite{sherry2020learning}, despite the lower sampling ratio. This is likely a consequence of the Huber approximation of the total variation that has to be used in \cite{sherry2020learning}.

\subsubsection{Eileen Collins}
We now proceed to the image of Eileen Collins as depicted in Figure \ref{subfig:ast-sparse-sanity1} and perform the same set of tasks as described in the previous section, i.e., we compute the element $\tilde{v}^K$ via \eqref{eq:palm} with zero-initialisation and for $J = \beta \| \cdot \|_1$ for which we choose $\beta = 0.24$ this time. The corresponding $\tilde{v}^K$, the mask of all non-zero coefficients and the mask of the 40000 largest Fourier coefficients (in magnitude) are visualised in Figure \ref{fig:ast-sparse-sc}. Please note that for $\beta = 0.24$ the number of non-zero coefficients of $\tilde{v}^K$ after $K = 1000$ iterations is 38262, which is comparable but even slightly lower than the 40000 non-zero coefficients that were used in the low-pass filter example in Section \ref{sec:fourier-subsampling-astronaut}. The number of non-zero coefficients of the approach in \cite{sherry2020learning} after 250 iterations with the same sparsity parameter $\beta$ is 59197, which is much higher and, like for the Shepp-Logan phantom, likely a result of the early termination of the algorithm.

\begin{figure}[H]
\centering
\subfloat[$\tilde{v}^K$]{\includegraphics[width=0.3\textheight]{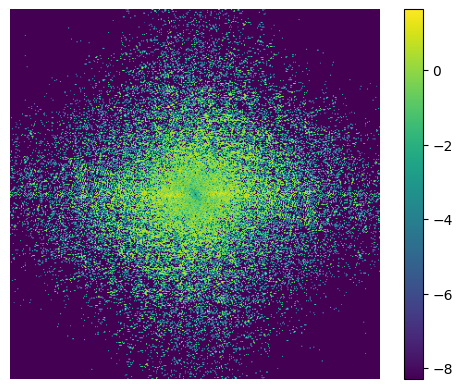}\label{subfig:ast-sparse-sc}}
\hspace{0.03\textwidth}
\subfloat[Sampling pattern]{\includegraphics[width=0.255\textheight]{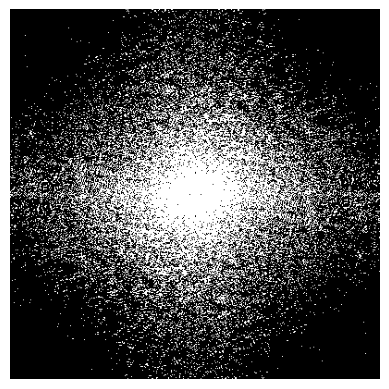}\label{subfig:ast-sparse-mask}}\\
\hspace{0.01\textwidth}
\subfloat[Learned sampling pattern with~\cite{sherry2020learning}]{\includegraphics[width=0.245\textheight]{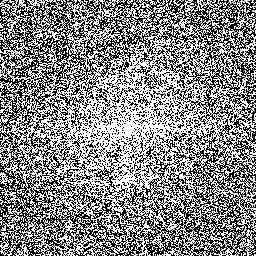} \label{subfig:ast-optimal-mask-comparison}}
\hspace{0.1\textwidth}
\subfloat[Largest Fourier coefficients]{\includegraphics[width=0.255\textheight]{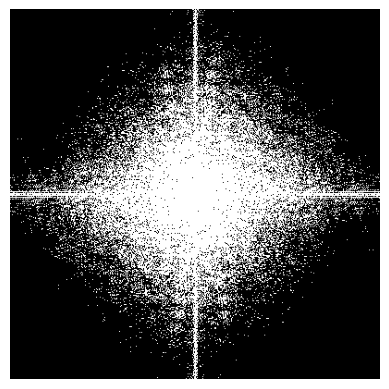} \label{subfig:ast-optimal-mask}}
\caption{Figure \ref{subfig:ast-sparse-sc} shows the (approximate) source condition element for the Eileen Collins image computed with \eqref{eq:palm}. Figure \ref{subfig:ast-sparse-mask} shows the corresponding mask with a ratio of 24\% non-zero vs zero coefficients. For comparison, Figure~\ref{subfig:ast-optimal-mask-comparison} shows the mask learned with the approach proposed in~\cite{sherry2020learning} with a ratio of 90.3\% and Figure \ref{subfig:ast-optimal-mask} shows the mask of the  Fourier coefficients with largest magnitude with a ratio of 25\%.}\label{fig:ast-sparse-sc}
\end{figure}

We see that the learned mask in Figure \ref{subfig:ast-sparse-mask} differs from that obtained from the Fourier coefficients with largest magnitude in Figure \ref{subfig:ast-optimal-mask}. In particular, the information that corresponds to coarse edges in the image $u^\dagger$ is less important for the total variation-based model, since it can interpolate this type of missing information rather well. Instead, more higher frequencies corresponding to texture information are being sampled as it is impossible for a total variation-based model to generate textures.

Given the mask, we estimate $v$ via \eqref{eq:coordinate-descent} and obtain $v^K$ for $K = 1000$. Note that the corresponding norm is $\| v^K \| \approx 302.47$, which is larger than the norm that we established in the low-pass filter setting. 
Subsequently, we perform another sanity-check and compute an approximation of \eqref{eq:variational-regularisation} for data $g_\alpha^K = \alpha v^K + S\mathcal{F}u^\dagger$ for $\alpha = 1/2$ via the PDHG method with the same initialisation and parameter configurations as described in Section \ref{sec:fourier-subsampling-astronaut} and visualise the results in Figure \ref{fig:ast-sparse-sc}.

\begin{figure}[t]
\centering
\subfloat[Ground truth $u^\dagger$]{
\includegraphics[width=0.33\textwidth]{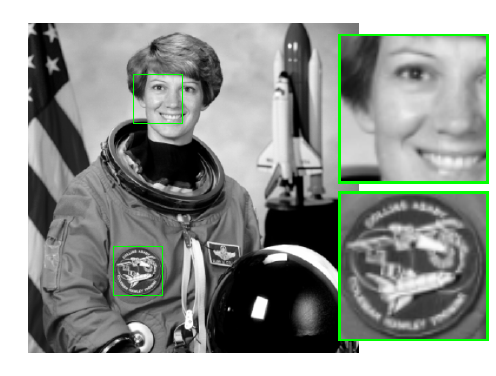}\label{subfig:ast-sparse-sanity1}}
\subfloat[Approximate solution $u^N$]{
\includegraphics[width=0.33\textwidth]{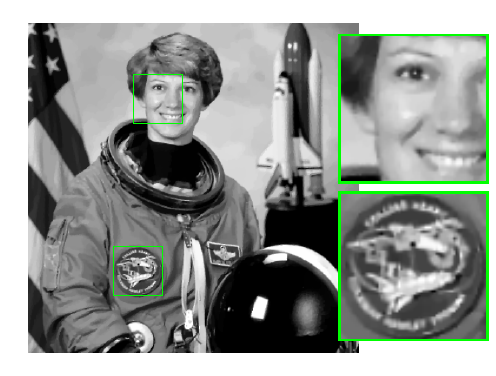}\label{subfig:ast-sparse-sanity2}} 
\subfloat[Projection $\mathcal{F}^{-1} S^\top S \mathcal{F} u^\dagger$]{
\includegraphics[width=0.33\textwidth]{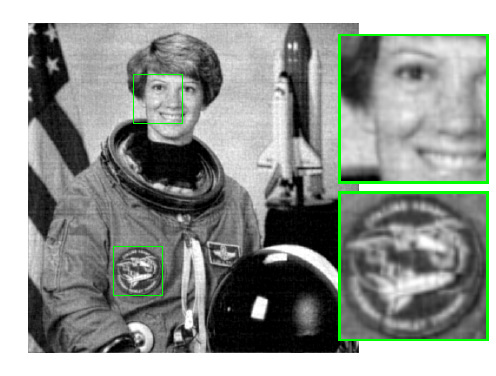}\label{subfig:ast-sparse-sanity3}} \\
\subfloat[Low-pass filtered reconstruction]{
\includegraphics[width=0.33\textwidth]{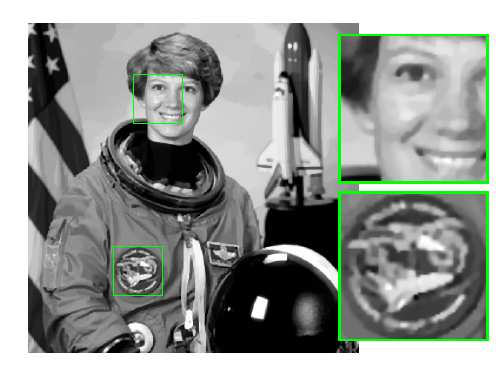}\label{subfig:ast-sparse-sanity4}}
\subfloat[Reconstruction with~\cite{sherry2020learning}]{
\includegraphics[width=0.33\textwidth]{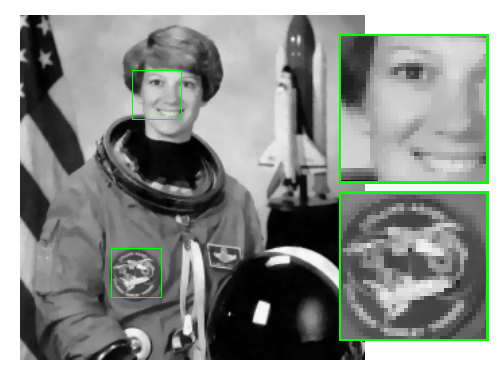}\label{subfig:ast-sparse-sanity5}}
\caption{In Figure~\ref{subfig:ast-sparse-sanity1} we show the image of Eileen Collins; in Figure~\ref{subfig:ast-sparse-sanity2} the approximation $u^N$ of \eqref{eq:variational-regularisation} with the learned mask and corresponding range data; in Figure~\ref{subfig:ast-sparse-sanity3} the projection $\mathcal{F}^{-1} S^\top S \mathcal{F} u^\dagger$; in Figure~\ref{subfig:ast-sparse-sanity4} the approximation $u^N$ of \eqref{eq:variational-regularisation} with the low-pass filter from Section \ref{sec:fourier-subsampling-astronaut}; in Figure~\ref{subfig:ast-sparse-sanity5} the reconstruction based on~\cite{sherry2020learning}. 
For all figures, a close-up of two region of interest is provided (green squares).}\label{fig:ast-sparse-sc}
\end{figure}

In comparison to Section \ref{sec:fourier-subsampling-astronaut}, we observe that we still have not found  a source condition element, which is not surprising. We do observe, however, that the recovery of \eqref{eq:variational-regularisation} with the range data for the learned sampling operator is much better compared to the recovery with the low-pass sampling operator. We observe several high-frequency features such as the specular highlights in the eyes that are not present in the low-pass filter reconstruction. Please note that the reconstruction with the learned sampling operator uses a smaller amount of samples than the reconstruction with the low-pass sampling operator.

\section{Conclusions \& outlook}\label{sec:conclusions}
We conclude this work with a brief summary of its main findings, and provide an outlook of research topics that we have not been able to address but that make for interesting future research.

\subsection{Conclusions}
In this paper, we have pursued the challenging task of estimating source condition elements as tools for the quantitative analysis of variational regularisation of linear inverse problems. Specifically, 
\begin{itemize}
    \item we considered a rather general Banach space setup, which encompasses a large class of regularisation methods of interest in applications;
    \item we reformulated the source conditions (and the closely related range conditions) as the solution of convex minimisation problems by means of tools from convex analysis;
    \item we provided iterative algorithms for the numerical approximation of source (and range) condition elements, which pave the way to quantitative error estimates;
    \item to demonstrate the performance of the proposed approach, we enclosed a significant set of numerical experiments for processing both 1D and 2D signals;
    \item we described how the introduced framework can provide insightful inspiration for novel approaches in optimal measurement design: in particular, the proposed procedure of optimal subsampling in the Fourier domain provided promising results.
\end{itemize}

\subsection{Outlook}
There are numerous research aspects that we have either addressed only briefly or have not addressed at all. One obvious aspect is that source and range conditions are not only relevant for convergence rates for variational regularisations but other regularisations such as iterative regularisations, too (cf. \cite{burger2007error,benning2018modern}). Hence, in order to quantify error estimates for such regularisations, identical strategies as proposed in this paper can be deployed. The same also holds true for regularisation methods that are based on more general data fidelity terms as described in \eqref{eq:VarRegu}, for which error estimates also rely on source or range conditions \cite{benning2011error}. 

Further, we did not look into stronger source conditions or variational source conditions as outlined in the introduction, but it should be straight forward to design convex minimisation problems similar to the ones presented in this work for the computation of, e.g., strong source condition elements. We also maintained focus solely on linear inverse problems, while source conditions play a pivotal role for convergence rates of nonlinear inverse problems, too. 

Another interesting direction for research is the computation of generalised eigenfunctions or singular vectors as addressed in the introduction. Suppose we are given a function $f$ and we would like to find a function $u_\lambda$ that satisfies $\lambda u_\lambda \in J(u_\lambda)$ such that it is close to $f$, then we can formulate the constrained minimisation problem
\begin{align*}
    \min_{u} \ B_J(u, (1 + \lambda) u) \qquad \text{subject to} \qquad \| u - f \| \leq \delta ,
\end{align*}
for hyperparameters $\lambda > 0$ and $\delta \geq 0$. If a solution $u_\lambda$ satisfies $B_J(u_\lambda, (1 + \lambda) u_\lambda) = 0$, we can conclude that it is a generalised eigenfunction with eigenvalue $\lambda$ and closest to $f$ within the ball of radius $\delta$.

Last but not least, we want to emphasise that extensions of the proposed minimisation problems can find applications in a wide range of data-driven inverse problems applications such as operator correction and learning data-driven regularisation functionals that we want to briefly describe in the following two sections. 

\subsubsection{Operator correction}\label{sec:operator-correction}
In analogy to the optimal sampling example presented in Section \ref{sec:optimal-sampling}, one can modify the proposed approaches to perform more general operator corrections beyond sampling. Taking the polynomial regression problem from Section \ref{subsec:1Dpoly} as an example, one can consider the following extension of the classical LASSO approach:
\begin{align*}
    u_\alpha \in \argmin_{u \in \ell^2 \cap \ell^1} \left\{ \frac12 \left\| A \left( Ku - f^\delta \right) \right\|^2 + \alpha \| u \|_1 \right\} ,
\end{align*}
where $\ell^2 \cap \ell^1$ denotes the intersection of the spaces $\ell^2$ and $\ell^1$. In this example, the goal is to estimate not only the source condition element $v$, but to also estimate a pre-conditioner matrix $A$. We could do so by formulating the non-convex optimisation problem
\begin{align*}
    \min_{A, \{ v_i \}_{i = 1}^s} \left[ \sum_{i = 1}^s B_{\| \cdot \|_1}\left(u_i^\dagger, u_i^\dagger + K^\top A^\top v_i\right) + \chi_{\| \cdot \| \leq \beta}(v_i) \right] ,
\end{align*}
where $\chi_{\| \cdot \| \leq \beta}$ denotes the characteristic function that for an argument $v$ is zero if $\| v \| \leq \beta$ and infinity otherwise. Here we minimise an empirical risk for $s$ vectors of polynomial coefficients $\{ u^\dagger_i \}_{i = 1}^s$ subject to the constraint that the norm of the source condition elements $v_i$ should not exceed the threshold $\beta > 0$. If the data is representative, minimising this empirical risk can help  preconditioning the forward model and the data to lower the norms of the corresponding source condition elements, which in return ensures better convergence rates of the solution of $u_\alpha$ towards true coefficients $u^\dag$.

\subsubsection{Construction of data-driven regularisations}

One can extend the findings from Section \ref{sec:part-a} to design data-driven variational regularisation operators $R_\alpha$ with convergence rate $\mathcal{O}(\delta)$ and favourable error amplification constant $\| v \|$ in the error estimate \eqref{eq:error-estimate-linear}. In order to do so, we can make the assumption that we have a parametrised variational regularisation operator $R_\alpha$ of the form \eqref{eq:variational-regularisation} with a regularisation function of the form $J(u) = H(Au + b)$. We can then aim at estimating the linear operator $A$ and the function $b$ by minimising an empirical risk of the form
\begin{align*}
    \frac{1}{s} \sum_{i = 1}^s \left[ \Gamma(v_i) + \frac12 \| K^\ast v_i - A^\ast q^\dagger_i \|^2 + B_H\left(A u^\dagger_i + b, q^\dagger_i + A u^\dagger_i + b\right) \right] ,
\end{align*}
for some suitable regularisation function $\Gamma$. In order to keep the error amplification constants $\| v_i \|$ bounded, one intuitive choice for $\Gamma$ is 
\begin{align*}
    \Gamma(v) = \chi_{\| \cdot \| \leq \beta}(v) \colon = \begin{cases} 0 & \text{if } \| v \| \leq \beta \\ \infty & \text{if } \| v \| > \beta \end{cases} ,
\end{align*}
for a positive constant $\beta > 0$, similar to the operator correction example in Section \ref{sec:operator-correction}. Then one can formulate the constrained, non-convex minimisation problem
\begin{align*}
    \min_{\{ v_i \}_{i = 1}^s, \{ q^\dagger_i \}_{i = 1}^s, A, b} \left\{ \sum_{i = 1}^s \left[ \chi_{\| \cdot \| \leq \alpha}(v_i) + \frac12 \| K^\ast v_i - A^\ast q^\dagger_i \|^2 + B_H\left(A u^\dagger_i + b, q^\dagger_i + A u^\dagger_i + b\right) \right] \right\} .
\end{align*}
Similar to the idea proposed in Section \ref{sec:operator-correction}, minimising such an empirical risk can help identifying a suitable operator $A$ (and shift $b$) to lower the norms of the corresponding source condition elements, which in return ensures better convergence rates of the solution of the variational regularisation method towards the solution $u^\dag$ of the inverse problem \eqref{eq:InvProbl}.

\section*{Acknowledgements}
The authors would like to thank the Isaac Newton Institute for Mathematical Sciences, Cambridge, for support and hospitality during the programme `Mathematics of Deep Learning' where work on this paper was undertaken. This work was supported by EPSRC Grant No. EP/R014604/1. Also INdAM-GNCS, INdAM-GNAMPA are acknowledged. MB acknowledges support from the Alan Turing Institute. LR was supported by the Air Force Office of Scientific Research under award number FA8655-20-1-7027, and acknowledges the support of
Fondazione Compagnia di San Paolo. Research is also partly funded by PNRR - M4C2 - Investimento 1.3. Partenariato Esteso PE00000013 - ``FAIR - Future Artificial Intelligence Research" - Spoke 8 ``Pervasive AI", which is funded by the European Commission under the NextGeneration EU programme.
DR acknowledges support from EPSRC grant EP/R513106/1.

\section*{Data Availability Statement}
The Python codes for this paper are available at \url{https://github.com/reacho/source-condition}.

\bibliographystyle{amsplain}
\bibliography{bibliography-2}

\providecommand{\bysame}{\leavevmode\hbox to3em{\hrulefill}\thinspace}
\providecommand{\MR}{\relax\ifhmode\unskip\space\fi MR }
\providecommand{\MRhref}[2]{%
  \href{http://www.ams.org/mathscinet-getitem?mr=#1}{#2}
}
\providecommand{\href}[2]{#2}
\begin{thebibliography}{10}

\bibitem{beck2013convergence}
Amir Beck and Luba Tetruashvili, \emph{On the convergence of block coordinate
  descent type methods}, SIAM Journal on Optimization \textbf{23} (2013),
  no.~4, 2037--2060.

\bibitem{benning2011error}
Martin Benning and Martin Burger, \emph{Error estimates for general
  fidelities}, Electronic Transactions on Numerical Analysis \textbf{38}
  (2011), no.~44-68, 77.

\bibitem{benning2013ground}
\bysame, \emph{Ground states and singular vectors of convex variational
  regularization methods}, Methods and Applications of Analysis \textbf{20}
  (2013), no.~4, 295--334.

\bibitem{benning2018modern}
\bysame, \emph{Modern regularization methods for inverse problems}, Acta
  Numerica \textbf{27} (2018), 1--111.

\bibitem{benning2021bregman}
Martin Benning and Erlend~Skaldehaug Riis, \emph{Bregman methods for
  large-scale optimisation with applications in imaging}, Handbook of
  Mathematical Models and Algorithms in Computer Vision and Imaging:
  Mathematical Imaging and Vision (2021), 1--42.

\bibitem{bolte2014proximal}
J{\'e}r{\^o}me Bolte, Shoham Sabach, and Marc Teboulle, \emph{Proximal
  alternating linearized minimization for nonconvex and nonsmooth problems},
  Mathematical Programming \textbf{146} (2014), no.~1, 459--494.

\bibitem{bozorgnia2020infinity}
Farid Bozorgnia, Leon Bungert, and Daniel Tenbrinck, \emph{The infinity
  laplacian eigenvalue problem: reformulation and a numerical scheme}, arXiv
  preprint arXiv:2004.08127 (2020).

\bibitem{bregman1967relaxation}
Lev~M Bregman, \emph{The relaxation method of finding the common point of
  convex sets and its application to the solution of problems in convex
  programming}, USSR Computational Mathematics and Mathematical Physics
  \textbf{7} (1967), no.~3, 200--217.

\bibitem{bubba2021convex}
Tatiana~A Bubba, Martin Burger, Tapio Helin, and Luca Ratti, \emph{Convex
  regularization in statistical inverse learning problems}, Inverse Problems
  and Imaging \textbf{17} (2023), no.~6, 1193--1225.

\bibitem{bubba2022shearlet}
Tatiana~A Bubba and Luca Ratti, \emph{Shearlet-based regularization in
  statistical inverse learning with an application to x-ray tomography},
  Inverse Problems \textbf{38} (2022), no.~5, 054001.

\bibitem{bungert2021nonlinear}
Leon Bungert, Ester Hait-Fraenkel, Nicolas Papadakis, and Guy Gilboa,
  \emph{Nonlinear power method for computing eigenvectors of proximal operators
  and neural networks}, SIAM Journal on Imaging Sciences \textbf{14} (2021),
  no.~3, 1114--1148.

\bibitem{burger2015spectral}
Martin Burger, Lina Eckardt, Guy Gilboa, and Michael Moeller, \emph{Spectral
  representations of one-homogeneous functionals}, Scale Space and Variational
  Methods in Computer Vision: 5th International Conference, SSVM 2015,
  L{\`e}ge-Cap Ferret, France, May 31-June 4, 2015, Proceedings, Springer,
  2015, pp.~16--27.

\bibitem{burger2016spectral}
Martin Burger, Guy Gilboa, Michael Moeller, Lina Eckardt, and Daniel Cremers,
  \emph{Spectral decompositions using one-homogeneous functionals}, SIAM
  Journal on Imaging Sciences \textbf{9} (2016), no.~3, 1374--1408.

\bibitem{burger2018large}
Martin Burger, Tapio Helin, and Hanne Kekkonen, \emph{Large noise in
  variational regularization}, Transactions of Mathematics and its Applications
  \textbf{2} (2018), no.~1, tny002.

\bibitem{burger2004convergence}
Martin Burger and Stanley Osher, \emph{Convergence rates of convex variational
  regularization}, Inverse Problems \textbf{20} (2004), no.~5, 1411.

\bibitem{burger2007error}
Martin Burger, Elena Resmerita, and Lin He, \emph{Error estimation for
  {B}regman iterations and inverse scale space methods in image restoration},
  Computing \textbf{81} (2007), no.~2, 109--135.

\bibitem{candes2006robust}
Emmanuel~J Cand{\`e}s, Justin Romberg, and Terence Tao, \emph{Robust
  uncertainty principles: Exact signal reconstruction from highly incomplete
  frequency information}, IEEE Transactions on Information Theory \textbf{52}
  (2006), no.~2, 489--509.

\bibitem{chambolle2004algorithm}
Antonin Chambolle, \emph{An algorithm for total variation minimization and
  applications}, Journal of Mathematical Imaging and Vision \textbf{20} (2004),
  89--97.

\bibitem{chambolle2016introduction}
Antonin Chambolle and Thomas Pock, \emph{An introduction to continuous
  optimization for imaging}, Acta Numerica \textbf{25} (2016), 161--319.

\bibitem{chavent1997regularization}
Guy Chavent and Karl Kunisch, \emph{Regularization of linear least squares
  problems by total bounded variation}, ESAIM: Control, Optimisation and
  Calculus of Variations \textbf{2} (1997), 359--376.

\bibitem{daubechies2004iterative}
Ingrid Daubechies, Michel Defrise, and Christine De~Mol, \emph{An iterative
  thresholding algorithm for linear inverse problems with a sparsity
  constraint}, Communications on Pure and Applied Mathematics \textbf{57}
  (2004), no.~11, 1413--1457.

\bibitem{donoho1992superresolution}
David~L Donoho, \emph{Superresolution via sparsity constraints}, SIAM Journal
  on Mathematical Analysis \textbf{23} (1992), no.~5, 1309--1331.

\bibitem{ekeland1999convex}
Ivar Ekeland and Roger Temam, \emph{Convex analysis and variational problems},
  SIAM, 1999.

\bibitem{engl1989convergence}
Heinz~W Engl, Karl Kunisch, and Andreas Neubauer, \emph{Convergence rates for
  {T}ikhonov regularisation of non-linear ill-posed problems}, Inverse Problems
  \textbf{5} (1989), no.~4, 523.

\bibitem{engl1996regularization}
Heinz~Werner Engl, Martin Hanke, and Andreas Neubauer, \emph{Regularization of
  inverse problems}, vol. 375, Springer Science \& Business Media, 1996.

\bibitem{flemming2012generalized}
Jens Flemming, \emph{Generalized {T}ikhonov regularization and modern
  convergence rate theory in {B}anach spaces}, Shaker-Verlag, 2012.

\bibitem{gilboa2014nonlinear}
Guy Gilboa, \emph{Nonlinear band-pass filtering using the tv transform}, 2014
  22nd European Signal Processing Conference (EUSIPCO), IEEE, 2014,
  pp.~1696--1700.

\bibitem{gilboa2014total}
\bysame, \emph{A total variation spectral framework for scale and texture
  analysis}, SIAM Journal on Imaging Sciences \textbf{7} (2014), no.~4,
  1937--1961.

\bibitem{gilboa2018nonlinear}
\bysame, \emph{Nonlinear eigenproblems in image processing and computer
  vision}, Springer, 2018.

\bibitem{gilboa2016nonlinear}
Guy Gilboa, Michael Moeller, and Martin Burger, \emph{Nonlinear spectral
  analysis via one-homogeneous functionals: overview and future prospects},
  Journal of Mathematical Imaging and Vision \textbf{56} (2016), 300--319.

\bibitem{grasmair2008sparse}
Markus Grasmair, Markus Haltmeier, and Otmar Scherzer, \emph{Sparse
  regularization with $\ell^q$ penalty term}, Inverse Problems \textbf{24}
  (2008), no.~5, 055020.

\bibitem{grasmair2011necessary}
Markus Grasmair, Otmar Scherzer, and Markus Haltmeier, \emph{Necessary and
  sufficient conditions for linear convergence of $\ell^1$-regularization},
  Communications on Pure and Applied Mathematics \textbf{64} (2011), no.~2,
  161--182.

\bibitem{hein2009approximate}
Torsten Hein and Bernd Hofmann, \emph{Approximate source conditions for
  nonlinear ill-posed problems—chances and limitations}, Inverse Problems
  \textbf{25} (2009), no.~3, 035003.

\bibitem{hofmann2006approximate}
Bernd Hofmann, D~D{\"u}velmeyer, and Klaus Krumbiegel, \emph{Approximate source
  conditions in {T}ikhonov regularization-new analytical results and some
  numerical studies}, Mathematical Modelling and Analysis \textbf{11} (2006),
  no.~1, 41--56.

\bibitem{hofmann2007convergence}
Bernd Hofmann, Barbara Kaltenbacher, Christiane P\"{o}schl, and Otmar Scherzer,
  \emph{A convergence rates result for {T}ikhonov regularization in {B}anach
  spaces with non-smooth operators}, Inverse Problems \textbf{23} (2007),
  no.~3, 987.

\bibitem{hohage2019optimal}
Thorsten Hohage and Philip Miller, \emph{Optimal convergence rates for sparsity
  promoting wavelet-regularization in {B}esov spaces}, Inverse Problems
  \textbf{35} (2019), no.~6, 065005.

\bibitem{hohage2017characterizations}
Thorsten Hohage and Frederic Weidling, \emph{Characterizations of variational
  source conditions, converse results, and maxisets of spectral regularization
  methods}, SIAM Journal on Numerical Analysis \textbf{55} (2017), no.~2,
  598--620.

\bibitem{lions1979splitting}
Pierre-Louis Lions and Bertrand Mercier, \emph{Splitting algorithms for the sum
  of two nonlinear operators}, SIAM Journal on Numerical Analysis \textbf{16}
  (1979), no.~6, 964--979.

\bibitem{mukherjee2021learning}
Subhadip Mukherjee, Carola-Bibiane Sch{\"o}nlieb, and Martin Burger,
  \emph{Learning convex regularizers satisfying the variational source
  condition for inverse problems}, NeurIPS 2021 Workshop on Deep Learning and
  Inverse Problems.

\bibitem{mumford1989optimal}
David~Bryant Mumford and Jayant Shah, \emph{Optimal approximations by piecewise
  smooth functions and associated variational problems}, Communications on Pure
  and Applied Mathematics (1989).

\bibitem{nesterov1983method}
Yurii Nesterov, \emph{A method of solving a convex programming problem with
  convergence rate $\mathcal{O}\bigl(\frac1{k^2}\bigr)$}, Dokl. Akad. Nauk
  USSR, vol. 269, 1983, pp.~543--547.

\bibitem{nossek2018flows}
Raz~Z Nossek and Guy Gilboa, \emph{Flows generating nonlinear eigenfunctions},
  Journal of Scientific Computing \textbf{75} (2018), 859--888.

\bibitem{ramlau2010convergence}
Ronny Ramlau and Elena Resmerita, \emph{Convergence rates for regularization
  with sparsity constraints}, Electronic Transactions on Numerical Analysis
  \textbf{37} (2010), 87--104.

\bibitem{resmerita2005regularization}
Elena Resmerita, \emph{Regularization of ill-posed problems in {B}anach spaces:
  convergence rates}, Inverse Problems \textbf{21} (2005), no.~4, 1303.

\bibitem{rudin1992nonlinear}
Leonid~I. Rudin, Stanley Osher, and Emad Fatemi, \emph{Nonlinear total
  variation based noise removal algorithms}, Physica D: Nonlinear Phenomena
  \textbf{60} (1992), no.~1-4, 259--268.

\bibitem{scherzer2009variational}
Otmar Scherzer, Markus Grasmair, Harald Grossauer, Markus Haltmeier, and Frank
  Lenzen, \emph{Variational methods in imaging}, Springer New York, 2009.

\bibitem{schmidt2018inverse}
Marie~Foged Schmidt, Martin Benning, and Carola-Bibiane Sch{\"o}nlieb,
  \emph{Inverse scale space decomposition}, Inverse Problems \textbf{34}
  (2018), no.~4, 045008.

\bibitem{schuster2012regularization}
Thomas Schuster, Barbara Kaltenbacher, Bernd Hofmann, and Kamil~S Kazimierski,
  \emph{Regularization methods in {B}anach spaces}, vol.~10, Walter de Gruyter,
  Berlin, Boston, 2012.

\bibitem{sherry2020learning}
Ferdia Sherry, Martin Benning, Juan~Carlos De~los Reyes, Martin~J Graves, Georg
  Maierhofer, Guy Williams, Carola-Bibiane Sch{\"o}nlieb, and Matthias~J
  Ehrhardt, \emph{Learning the sampling pattern for mri}, IEEE Transactions on
  Medical Imaging \textbf{39} (2020), no.~12, 4310--4321.

\bibitem{tautenhahn1998optimality}
Ulrich Tautenhahn, \emph{Optimality for ill-posed problems under general source
  conditions}, Numerical Functional Analysis and Optimization \textbf{19}
  (1998), no.~3-4, 377--398.

\bibitem{tibshirani1996regression}
Robert Tibshirani, \emph{Regression shrinkage and selection via the lasso},
  Journal of the Royal Statistical Society: Series B (Methodological)
  \textbf{58} (1996), no.~1, 267--288.

\bibitem{tikhonov1943stability}
Andrey~N. Tikhonov, \emph{On the stability of inverse problems}, Dokl. Akad.
  Nauk SSSR, vol.~39, 1943, pp.~195--198.

\bibitem{tihonov1963solution}
\bysame, \emph{Solution of incorrectly formulated problems and the
  regularization method}, Soviet Math. \textbf{4} (1963), 1035--1038.

\bibitem{wang2022lifted}
Xiaoyu Wang and Martin Benning, \emph{Lifted bregman training of neural
  networks}, Journal of Machine Learning Research \textbf{24} (2023), no.~232,
  1--51.

\bibitem{wright2015coordinate}
Stephen~J Wright, \emph{Coordinate descent algorithms}, Mathematical
  Programming \textbf{151} (2015), no.~1, 3--34.

\bibitem{doi:10.1137/120887795}
Yangyang Xu and Wotao Yin, \emph{A block coordinate descent method for
  regularized multiconvex optimization with applications to nonnegative tensor
  factorization and completion}, SIAM Journal on Imaging Sciences \textbf{6}
  (2013), no.~3, 1758--1789.

\end{thebibliography}

\end{document}